\documentclass[12pt]{amsart}
\usepackage{amsthm,amsmath,a4,amssymb,url,verbatim,enumerate}

\newtheorem{thm}{Theorem}[section]
\newtheorem{cor}[thm]{Corollary}
\newtheorem{lem}[thm]{Lemma}
\newtheorem{prop}[thm]{Proposition}
\theoremstyle{definition}
\newtheorem{defn}[thm]{Definition}
\theoremstyle{remark}
\newtheorem{rem}[thm]{Remark}
\newtheorem{ex}[thm]{Example}
\numberwithin{equation}{section}

\newcommand{\Hom}{\textnormal{Hom}}
\newcommand{\End}{\textnormal{End}}
\newcommand{\Aut}{\textnormal{Aut}}

\newcommand{\Ker}{\textnormal{Ker}}

\newcommand{\lp}{(\!(}
\newcommand{\rp}{)\!)}

\newcommand{\N}{\mathbf{N}}
\newcommand{\C}{\mathbf{C}}
\newcommand{\K}{\mathbf{K}}
\newcommand{\Z}{\mathbf{Z}}
\newcommand{\R}{\mathbf{R}}
\newcommand{\Q}{\mathbf{Q}}
\newcommand{\F}{\mathbf{F}}
\DeclareMathOperator\Ell{W}
\DeclareMathOperator\Con{C}
\newcommand{\grp}{\mathrm{grp}}
\DeclareMathOperator\Dist{\mathcal{E}}
\DeclareMathOperator\GL{GL}

\begin{document}

\address{CNRS -- Laboratoire Math\'ematique Jean Leray -- Universit\'e de Nantes}

\email{ycornulier@univ-nantes.fr}

\title[LC modules over abelian groups and metabelian groups]{Locally compact modules over abelian groups and compactly generated metabelian groups}
\author{Yves Cornulier}%
\date{November 19, 2023}


\begin{abstract}
We perform a general study of the structure of locally compact modules over compactly generated abelian groups. We obtain a d\'evissage result for such modules of the form ``compact-by-sheer-by-discrete", and then study more specifically the sheer part. The main typical example of a sheer module is a polycontractable module, i.e., a finite direct product of modules, each of which is contracted by some group element. We show that every sheer module has a ``large" polycontractable submodule, in a suitable sense.

We apply this to the study of compactly generated metabelian groups. For instance, we prove that they always have a maximal compact normal subgroup, and we extend the Bieri-Strebel characterization of compactly presentable metabelian groups from the discrete case to this more general setting.
\end{abstract}
\maketitle

\section{Introduction}

We fix a locally compact abelian (LCA) group $A$, which will often be assumed to be compactly generated. General structure of LCA groups ensures that every compactly generated LCA group has a cocompact lattice isomorphic to $\Z^d$ for some $d$. Therefore, the reader can assume $A=\Z^d$; this is not a significant restriction.

We consider locally compact (LC) $A$-modules $M$, that is, $M$ is an LCA-group endowed with a continuous homomorphism $A\to\Aut(M)$ (continuity is equivalent to continuity of the action $A\times M\to M$). We aim at providing general results on the structure of such modules $M$. They will then be illustrated in the study of compactly generated LC metabelian groups. One extreme case is when $M$ is discrete. In this case, it can be handled with classical commutative algebra. The ``opposite" case is when $M$ is compact: such modules then inherit a description, namely as Pontryagin duals of discrete modules. Recall that the Pontryagin dual of an LCA group $M$ is $\hat{M}=\Hom(M,\R/\Z)$ (continuous group homomorphisms), and that this defines a contravariant self-equivalence of the category of LCA groups, notably exchanging discrete and compact modules. Functoriality ensures that it is compatible with group actions, and thus it induces a contravariant self-equivalence of the category of $A$-modules (there is a minor continuity verification for the action, see \S\ref{cont-pont}). In particular, denoting by $N^\bot$ the orthogonal $\{\chi\in\hat{M}:\chi|_N=0\}$ for a closed submodule $N$ of $M$, orthogonality is an order-reversing isomorphism between closed submodules of $M$ and those of $\hat{M}$. 

The purpose of the forthcoming definitions is to ``isolate" the ``essential part" of $M$ that is ``in between" ---~roughly speaking, that has ``few" compact submodules and discrete quotients ---~and to do it precisely, so as to avoid such quotation marks (this will be done in Definition \ref{d-sheer}, Theorem \ref{fsheer}). The first definition in this purpose is thus:

\begin{defn}
The {\bf polycompact radical} $\Ell(M)$ (or $\Ell_A(M)$ in case of ambiguity) of $M$ is the sum of all compact submodules. The {\bf codiscrete radical} $\Omega(M)$ is the intersection of all open submodules.
\end{defn}

Note that modding out by the closure of the polycompact radical, and taking the codiscrete radical are ``exchanged" by Pontryagin duality (see Proposition \ref{ellomega} for details). A first fact is the following:

\begin{thm}\label{eclo}
If $A$ is compactly generated, then $\Ell(M)$ is closed for every LC-module $M$, and every compact subset of $\Ell(M)$ is contained in a compact submodule.
\end{thm}

This is easily deduced from a group-theoretic fact, namely, in a compactly generated LC group, the subgroup generated by all compact normal subgroups is closed. The latter is a nontrivial fact, proved in \cite{C15}, which relies on work of Trofimov, who essentially proved it for totally disconnected compactly generated LC groups. See \S\ref{pric}.


We now alter these definitions because of the following shortcomings:\begin{itemize}\item $\Ell(M)$ need not be compact;\item $\Ell(M)$ need not be contained in $\Omega(M)$ (e.g., if $M$ is finite, $\Omega(M)=\{0\}$ and $\Ell(M)=M$).\end{itemize}

\begin{defn}\label{d-sheer}
A module $M$ is {\bf compact parafinite} if it is compact with no nonzero finite quotient. Define the {\bf reduced polycompact radical} $\Ell^\flat(M)$ as the sum of all compact parafinite submodules of $M$.

A module $M$ is {\bf purely discrete} if it is discrete with no nonzero finite submodule. Define the {\bf augmented codiscrete radical} $\Omega^\sharp(M)$ as the intersection of all open submodules with purely discrete quotient.

A module $M$ is {\bf sheer} if $\Ell^\flat(M)=\{0\}$ and $\Omega^\sharp(M)=M$.
\end{defn}

Clearly, $0\subseteq \Ell^\flat(M)\subset \Ell(M)$ and $\Omega(M)\subseteq \Omega^\sharp(M)\subseteq M$; note that the latter are also exchanged by orthogonality in Pontryagin duality: $\Ell(M)^\bot=\Omega(\hat{M})$ and $\Ell^\flat(M)^\bot=\Omega^\sharp(\hat{M})$. In particular, a module is sheer if and only if its Pontryagin dual is sheer. A first observation is the following.

\begin{thm}[Canonical filtration with sheer subquotient]\label{fsheer}
We have the inclusions $0\subseteq \Ell^\flat(M)\subseteq \Omega^\sharp(M)\subseteq M$. If $A$ is compactly generated, then $\Ell^\flat(M)$ is compact parafinite, $M/\Omega^\sharp(M)$ is purely discrete, and the subquotient $\Omega^\sharp(M)/\Ell^\flat(M)$ is sheer.
\end{thm}

See \S\ref{cfss} for the proof. The first assertion is quite straightforward, while the second one makes use of Theorem \ref{eclo}.

Theorem \ref{fsheer} reduces most of the study of modules to that of sheer modules, so we proceed to describe structural results about them. Central in their description is the notion of contracting automorphism. Recall that an automorphism $\alpha$ of a LC group $G$ is {\bf contracting} if $\alpha^n(x)$ tends to 1, uniformly for $x$ in compact subsets of $G$.  

\begin{defn}
A LC $A$-module $M$ is said to be {\bf contractable} if there exists $\alpha\in A$ acting as a contracting automorphism of the LCA-group $M$. It is said to be {\bf polycontractable} if it splits as a finite topological direct sum of contractable modules (not necessarily with the same contracting element).

A LC $A$-module is said to be {\bf amorphic} if it is both sheer and compact-by-discrete. This means that it is profinite-by-(discrete locally finite). (Here both ``locally finite" and ``profinite" refer to the module structure and not just the underlying LCA group.)

A LC $A$-module is said to be {\bf Euclidean} if the underlying LCA group is isomorphic to $\R^d$ for some $d$. It is said to be
{\bf distal Euclidean} if it is Euclidean and the action of $A$ is by matrices with only complex eigenvalues of modulus 1.
\end{defn}

All these notions are invariant under Pontrayagin duality. For instance, if $p$ is prime and $M=\Q_p$ with some action $i:A\to\Q_p^*$, then $M$ is sheer, and either $i(A)$ consists of elements of norm one and $M$ is amorphic, or else $M$ is contractable.

Polycontractable modules are sheer and they constitute the main source of sheer modules, and, in the sense, the most ``interesting" ones. Indeed, it readily follows from the definition that $\Ell(M)=\{0\}$ and $\Omega(M)=M$ when $M$ is contractable.
Also, Euclidean modules and amorphic modules are sheer. The following theorem indicates that each sheer module has a canonical ``large" polycontractable submodule.

\begin{thm}\label{sheerdec}
Suppose that $A$ is compactly generated. Every sheer module has a unique decomposition $P\oplus E\oplus V$, with $P$ polycontractable, $E$ amorphic, $V$ distal Euclidean.
\end{thm}

See \S\ref{desheer}. The uniqueness part is not difficult. The existence part relies on the case $A=\Z$, where it takes a more precise form, see Lemma \ref{singledec}. The latter makes use of Willis theory \cite{Wil}, which deals with automorphisms of (possibly non-abelian) totally disconnected locally compact groups. Here we take advantage of the module being an abelian group to obtain these more precise results.

\begin{cor}\label{wcompact-i}
Suppose that $A$ is compactly generated. Then for every compactly generated LC $A$-module $M$, the polycompact radical $\Ell(M)$ is compact.
\end{cor}

\begin{ex}[Being sheer does not pass to submodules or quotients]\label{exzp1}
Here are some examples of polycontractable, and hence sheer, modules $M$ over the group $A=\Z$ with a non-sheer closed submodule $N$, namely a cocompact lattice. In each case, the quotient $M/N$ is compact parafinite and hence is not sheer either. The first example is totally disconnected, and the third is Euclidean.
\begin{enumerate}
\item The module $\F_p\lp t\rp\times\F_p\lp t^{-1}\rp$, where the positive generator of $\Z$ acts on both factors by multiplication by $t$, is polycontractable. The module $\F_p[t,t^{-1}]$ can be viewed as a discrete cocompact submodule of the former, by the obvious natural diagonal embedding.


\item Similarly, the module $\R\times\Q_p$, where the positive generator acts on both factors by multiplication by $p$, is polycontractable. It admits the diagonal embedding of $\Z[1/p]$ as a discrete submodule. In this case, the quotient, called ``solenoid", is connected.

\item Similarly, again with $A=\Z$ acting on $\R^2$ through powers of the matrix $\begin{pmatrix}2& 1\\1&1\end{pmatrix}$, $\R^2$ is a sheer module, but its submodule $\Z^2$ is not, since it is purely discrete.
\end{enumerate}
\end{ex}

\begin{thm}\label{montotypic-i}
Suppose that $A$ is compactly generated. Let $M$ be a sheer LC $A$-module. Then $M$ has a canonical decomposition $\big(\bigoplus_{i=1}^nP_i\big)\oplus E\oplus V$, with $P_i\neq\{0\}$ contractable monotypic (in the sense that all its simple quotients are isomorphic), $E$ amorphic, $V$ distal Euclidean. Moreover every closed sheer submodule $N$ splits with respect this decomposition, namely $N=\big(\bigoplus_{i=1}^nN\cap P_i\big)\oplus (N\cap E)\oplus (N\cap V)$.
\end{thm}

Theorem \ref{montotypic-i} gives a reasonable understanding of sheer closed submodules of sheer modules. It mostly reduces to understanding submodules of contractable (monotypic) modules. These ones are tackled by the following (see Theorem \ref{afl1}).

\begin{thm}
Let $M$ be a contractable $A$-module. Let $R$ be the closure of the image of the group ring $\Z[A]$ in $\End_A(M)$. Then $R$ is an artinian commutative ring, $M$ is a module of finite length over $R$, and the closed $A$-submodules of $M$ are exactly the $R$-submodules of $M$ (i.e., these submodules are automatically closed).
\end{thm}

(Here, the group ring is defined regardless of the topology on $A$.) These modules are thus completely governed by abstract commutative algebra. Example \ref{exzp1} shows that there are non-sheer submodules of sheer modules occurring as cocompact lattices. The following theorem (see Theorem \ref{corenv}) shows that taking cocompact lattices is essentially the only way to do so.

\begin{thm}
Suppose that $A$ is compactly generated. Let $M$ be a sheer $A$-module, and $N$ a closed submodule. Then there are unique sheer closed submodules $N^\boxminus$, $N^\boxplus$ of $M$ with $N^\boxminus\subseteq N\subseteq N^\boxplus$, with $N/N^\boxminus$ a cocompact lattice in $N^\boxplus/N^\boxminus$.
\end{thm}

Let us now describe applications to metabelian locally compact groups (\S\ref{metab-cg}). 



\begin{thm}
Let $G$ be a compactly generated metabelian group. Then for every ascending sequence $(H_n)$ of closed normal subgroups, there exists $n$ such that $H_n$ is cocompact in $\overline{\bigcup_k H_k}$. In particular, 
\begin{enumerate}
\item\label{mxc} $G$ has a maximal compact normal subgroup;
\item every closed normal subgroup is compactly generated as normal subgroup.
\end{enumerate}
\end{thm}

Actually (\ref{mxc}) is proved directly, making uses of applications in both senses: groups to modules, then modules to groups. Namely, if $G$ is a compactly generated LC group, then $\Ell(G)$, the union of all compact normal subgroups, is closed. This fact is used in the study of modules (Theorem \ref{eclo}). Next structural results on LC modules give rise to the compactness result for modules, Corollary \ref{wcompact-i}, which in turns implies its analogue for groups, (\ref{mxc}) above.

We now describe a more elaborate application. Let $G$ be a compactly generated metabelian group. Let $M$ be a closed normal abelian subgroup such that $Q=G/M$ is abelian. For a character $\chi\in\Hom(Q,\R)$, define $Q_\chi=\{g\in Q:\chi(g)\ge 0\}$. Define the Bieri-Strebel invariant as
\[\Gamma(G)=\{0\}\cup\{\chi\in\Hom(Q,\R):\; M\text{ is not comp. gen. as }Q_\chi\text{-module}\}.\]

This was defined by Bieri and Strebel \cite{BS} when $G$ is discrete. They actually rather defined its complement. But the convention here is much more convenient, for the reason that, after identifying $\Hom(Q,\R)$ to a subspace of $\Hom(G,\R)$ it does not depend on the choice of $M$ (while the complement does: just consider $G=\Z$ and the choices of $M=\{0\}$ vs $M=\Z$ to see the issue with taking the open complement). The subset $\Gamma(G)$ shares properties with the a spectrum of an operator, or a set of weights. It is invariant under positive homotheties. The following theorem extends the case of discrete metabelian groups from \cite{BS}.

\begin{thm}\label{metab-cp}
Let $G$ be a compactly generated metabelian LC group. Then $G$ is compactly presentable if and only $\Gamma(G)$ contains no line through the origin, or equivalently does not contain any two opposite nonzero characters. 
\end{thm}
 
The main ``typically non-discrete" case of Theorem \ref{metab-cp} is when $M$ (closed normal abelian subgroup of $G$ with abelian quotient $Q=G/M$) is polycontractable, namely $M=\bigoplus_{i\in I}M_i$, with $M_i$ contractable indecomposable (hence either totally disconnected or connected). Let $\theta_i:Q\to\R$ be the logarithm of the homomorphism describing the action of $Q$ on the Haar measure of $M_i$. Let $J$ be the set of $i$ such that $M_i$ is totally disconnected. Then $\Gamma(M)=\{0\}\cup\bigcup_{j\in J}\R_{\ge 0}\theta_j$.

In the discrete case, $\Gamma$ is known to be a polyhedral cone, not necessarily reduced to a finite union of half-lines. Proving Thereom \ref{metab-cp} thus has two first special cases: the sheer case (easily reducing to the polycontractable case), and the discrete case (due to Bieri-Strebel). Then additional work is needed to ``put things together", relying on the previous structural results.

\medskip

{\bf Acknowledgement.} I thank P-E.\ Caprace and G.\ Willis for encouraging me to write this material down.

\section{Basic facts}

\subsection{Pontryagin duality}\label{cont-pont}

Continuity of the action defining a LC $A$-module $M$ is equivalent to continuity of $A\to\Aut_{\mathrm{LCA}}(M)$. Recall that the topology of the latter is defined by: $\alpha_i\to\alpha$ if the map $(\alpha_i,\alpha_i^{-1}):M^2\to M^2$ converges to $(\alpha,\alpha^{-1})$, uniformly on compact subsets of $M$. To check that continuity is preserved when passing to Pontryagin duality, we just need to check the following.

\begin{prop}
Let $M$ be an LCA group. Then the action of $\Aut(M)$ on the Pontryagin dual $\hat{M}=\Hom(M,\R/\Z)$ is continuous.
\end{prop}
\begin{proof}
It is enough to prove that if $(\alpha_i)$ is a net in $\Aut(M)$ converging to the identity, then so does $(\hat{\alpha}_i)$. We have, for $\chi\in \hat{M}$ and $x\in M$
\[\hat{\alpha}_i(\chi)(x)=\chi(\alpha_i^{-1}(x))\]
and thus 
\[\hat{\alpha}_i(\chi)(x)-\chi(x)=\chi(\alpha_i^{-1}(x)-x).\]
Then $\alpha_i^{-1}(x)-x$ converges to $0$, uniformly for $x$ in compact subsets. Hence $\chi(\alpha_i^{-1}(x)-x)$ converges to $0$, uniformly for $(x,\chi)$ in compact subsets. This means that $\hat{\alpha}_i(\chi)$ converges to $\chi$, uniformly for $\chi$ in compact subsets. In turn, this means that $\hat{\alpha}_i$ converges to the identity, uniformly on compact subsets. The same argument with $\alpha_i^{-1}$ yields that ${\hat{\alpha}_i}^{-1}$ also converges to the identity uniformly on compact subsets. So $\hat{\alpha}_i$ tends to the identity in the topology of $\Aut(\hat{M})$.
\end{proof}

Let us provide, as a sample, the following fact of Pontryagin duality.

\begin{prop}\label{ellomega}Let $M$ be an $A$-module. The Pontryagin dual of the inclusion map $\Ell(M)\to M$ is (canonically identified with) the quotient map $\hat{M}\to\hat{M}/\Omega(\hat{M})$. In $\hat{M}$, the orthogonal of $\Ell(M)$ is $\Omega(\hat{M})$.
\end{prop}
\begin{proof}
The above inclusion map $i:\Ell(M)\to M$ satisfies the following universal property: for every compact $A$-module $N$ and continuous homomorphism $f:N\to M$, there exists a unique $A$-module homomophism $g:N\to\Ell(M)$ such that $f=i\circ g$. Therefore, by Pontryagin duality, the projection map $p:\hat{M}\to\widehat{\Ell(M)}$ satisfies the following universal property: for every continuous discrete $A$-module $N$ (continuous meaning that the $A$-action is continuous, i.e., point stabilizers are open) and every continuous $A$-module homomorphism $f:\hat{M}\to N$, there exists a unique $A$-module homomorphism $g:\widehat{\Ell(M)}\to N$ such that $f=g\circ p$. But clearly the quotient $\hat{M}\to\hat{M}/\Omega(\hat{M)}$ also satisfies this universal property. The result follows.

Recall that the orthogonal of a submodule $N$ of $M$ is $N^\bot=\{\chi\in\hat{M}:\chi|_N=0\}$. There is a canonical isomorphism $\hat{N}\simeq \hat{M}/N^\bot$. The second fact therefore follows from the first. Alternatively, a submodule is compact if and only if its orthogonal is open. Orthogonality is an order-reversing isomorphism between the lattices of submodules of $M$ and $\hat{M}$. Then $\Ell(M)$, as supremum (in this lattice) of the set of all compact submodule of $M$, is mapped to the infimum of the set of all open submodules of $\hat{M}$, and this is $\Omega(\hat{M})$.
\end{proof}

\subsection{Lower sheer, upper sheer modules}

Since ``sheer" involves two conditions, it is useful to separate them.

\begin{defn}
An $A$-module $M$ is {\bf lower-sheer} if $\Ell^\flat(M)=\{0\}$, and {\bf upper-sheer} if $\Omega^\sharp(M)=M$. It is {\bf sheer} if is both lower-sheer and upper-sheer.
\end{defn}

Note that the classes of lower-sheer and upper-sheer $A$-modules are exchanged under Pontryagin duality.

Note that lower-sheer passes to closed submodules, and upper-sheer passes to quotients. However sheer passes neither to submodules nor quotients (Example \ref{exzp1}).

\subsection{The polycompact radical is closed}\label{pric}

We use the following

\begin{thm}\label{gwclo}
Let $G$ be a compactly generated locally compact group. Then the union $\Ell_\grp(G)$ of all compact normal subgroups of $G$, is closed, and every compact subset of $\Ell_\grp(G)$ is contained in a compact normal subgroup of $G$.
\end{thm}
\begin{proof}
This is \cite[Theorem 2.8]{C15}. Let us point out that the proof in \cite{C15}, published in 2015, has been slightly corrected in the 2017 arXiv version at \url{https://arxiv.org/abs/1306.4194}. Let us also mention that the main case, that of $G$ totally disconnected, was essentially due to Trofimov. More precisely, he obtained the case of cocompact subgroups of automorphisms of locally finite connected graphs. The totally disconnected case follows by the ``Cayley-Abels graph" construction \cite[Prop.\ 2.E.9]{CH}. 
\end{proof}


\begin{thm}\label{closed-pr}
If $A$ is compactly generated, then $\Ell(M)$ is closed for every LC $A$-module $M$, and every compact subset of $\Ell(M)$ is contained in a compact submodule.
\end{thm}
\begin{proof}

Let $N$ be the $A$-submodule generated by compact subset of $M$ with nonempty interior. Thus $N$ is an open compactly generated submodule of $M$. So $N\rtimes A$ is a compactly generated LC group, and hence $\Ell_\grp(N\rtimes A)$ is closed by Theorem \ref{gwclo}. We directly see from the definition that $\Ell_\grp(N\rtimes A)\cap N$ equals $\Ell(N)$. Thus $\Ell(N)$ is closed in $N$. It also follows from the definition that $\Ell(M)\cap N=\Ell(N)$. Thus the subgroup $\Ell(M)$ is locally closed, and hence it is closed. 

For the additional statement, let $K$ be a compact subset of $\Ell(M)$. Above, we can choose $N$ to contain $K$. Then the above reference also ensures that $K$ is contained in a compact normal subgroup $L$ of $N\rtimes A$. Then $L\cap N$ is a compact submodule of $M$ containing $K$.
\end{proof}

\subsection{Canonical filtration with sheer subquotient}\label{cfss}

\begin{thm}\label{cafi}
Suppose that $A$ is compactly generated. We have the inclusions $0\subseteq \Ell^\flat(M)\subseteq \Omega^\sharp(M)\subseteq M$. If $A$ is compactly generated, then $\Ell^\flat(M)$ is compact parafinite, $M/\Omega^\sharp(M)$ is purely discrete, and subquotient $\Omega^\sharp(M)/\Ell^\flat(M)$ is sheer.
\end{thm}

\begin{proof}
If $N_1$ is a compact parafinite submodule and $N_2$ is an open submodule with purely discrete quotient, the projection of $N_1$ on $M/N_2$ is compact discrete, hence finite; since $N_1$ is parafinite, it is trivial. This shows that $N_1\subseteq N_2$. Since this holds for all $N_1,N_2$, we deduce the inclusion $\Ell^\flat(M)\subseteq \Omega^\sharp(M)$.

Now assume that $A$ is compactly generated, and let us show that the reduced polycompact radical is compact for every LC $A$-module. By Pontryagin duality, openness of the augmented codiscrete radical follows for every LC $A$-module.

Define $N$ as the closure of $\Ell^\flat(M)$. Let $K$ be the closure of the submodule generated by some compact subset of $N$ with nonempty interior. By Proposition \ref{closed-pr}, $K$ is compact; by construction, $K$ is open in $N$. Since $N$ is generated by its compact parafinite submodules, this is also true for the discrete module $N/K$, which is thus generated by its finite parafinite submodules. Since a finite parafinite module is trivial, this means that $N=K$. Thus $N$ is compact, and being topologically generated by parafinite submodules, it is parafinite. Thus $N=\Ell^\flat(M)$. Hence $\Ell^\flat(M)$ is compact; it is clearly parafinite. By Pontryagin duality (applied to $\hat{M}$), we deduce that $\Omega^\flat(M)$ is open, with purely discrete quotient.

Given that $\Ell^\flat(M)$ is compact parafinite and the extension of two compact parafinite modules is compact parafinite, we see that $M/\Ell^\flat(M)$ is lower-sheer, and hence so is $\Omega^\sharp(M)/\Ell^\flat(M)$. By Pontryagin duality, the latter is therefore also upper-sheer.
\end{proof}

\subsection{Extensions of sheer modules}

\begin{prop}\label{extsheer}
An extension of two sheer (resp.\ lower-sheer, resp.\ upper-sheer) modules is also sheer (resp.\ lower-sheer, resp.\ upper-sheer).
\end{prop}
\begin{proof}
Upper-sheer means that every homomorphism to a purely discrete module is zero, and this clearly passes to extensions. The lower-sheer case follows by Pontryagin duality. The sheer case follows from the other two cases.
\end{proof}

\begin{prop}\label{omush}
Suppose that $A$ is compactly generated. The submodule $\Omega^\sharp(M)$ is the largest upper-sheer submodule of $M$, and $\Omega^\sharp(\Omega^\sharp(M))=\Omega^\sharp(M)$.

The quotient $M/\Ell^\beta(M)$ is the largest lower-sheer quotient module of $M$, and $\Ell^\beta(M/\Ell^\beta(M))=\{0\}$.
\end{prop}
\begin{proof}
Let $N$ be an upper-sheer closed submodule of $M$. Then the image of $N$ into every purely discrete quotient of $M$ is zero, so $N\subseteq \Omega^\sharp(M)$. Then $M/\Omega^\sharp(\Omega^\sharp(M))$ is extension of $\Omega^\sharp(M)/\Omega^\sharp(\Omega^\sharp(M))$ by $M/\Omega^\sharp(M)$, both of which are purely discrete (by Theorem \ref{cafi}), and hence is purely discrete. So $\Omega^\sharp(\Omega^\sharp(M)\subseteq \Omega^\sharp(M)$, whence the two are equal, which means that $\Omega^\sharp(M)$ itself is upper-sheer.

The second statement follows, by Pontryagin duality.
\end{proof}

\begin{prop}
Suppose that $A$ is compactly generated. Let $M$ be an LC module and $N$ a submodule. Then $\Omega^\sharp(N)\subseteq\Omega^\sharp(M)$.
\end{prop}
\begin{proof}
By Proposition \ref{omush}, $\Omega^\sharp(N)$ is upper-sheer. Hence the image of $\Omega^\sharp(N)$ in the purely discrete module $M/\Omega^\sharp(M)$ is zero, which means that $\Omega^\sharp(N)\subseteq\Omega^\sharp(M)$.
\end{proof}

\subsection{Structure of distal modules}

\begin{prop}[Conze-Guivarch \cite{CGu}]\label{cogu}
Let $G$ be a group. Let $V$ be a distal Euclidean $G$-module. If $V$ is irreducible (as real representation), then the image of $G$ in $\GL(V)$ has compact closure. Hence, in general, there is an upper-triangular decomposition such that the image of $G$ in $\GL$ of the diagonal blocks has compact closure.\qed
\end{prop}

The following consequence is well-known.

\begin{cor}\label{ccogu}
Let $G$ be a group and $V$ a distal Euclidean $G$-module. Then no $G$-orbit $\neq\{0\}$ accumulates at zero.
\end{cor}
\begin{proof}
Let by contradiction $X\neq\{0\}$ be an orbit accumulating at zero. Let $W$ be the span of the given orbit, and $W'$ an irreducible quotient of $W$. Then the image $X'$ of $X$ in $W'$ is also an orbit accumulating at zero, and spans $W'\neq\{0\}$. Since $G$ has a compact closure in $\GL(W')$ by Proposition \ref{cogu}, a nonzero orbit cannot accumulate at zero. This is a contradiction.  
\end{proof}

\subsection{No homomorphism between modules of different kinds}

Uniqueness in Theorem \ref{sheerdec} is based on the following lemma.

\begin{lem}\label{homzero}
Let $C$ be a polycontractable $A$-module, $E$ an amorphic one, $V$ a distal Euclidean one. Then every homomorphism between any two distinct elements of $\{C,E,V\}$ is zero.
\end{lem}
\begin{proof}
It is no restriction to assume that $C$ is contractable.

$\Hom(V,C)=0$. Let $f:V\to C$ be a homomorphism, and $H$ the closure of the image; it is also contractable. Also, $H$ is connected. So, by general structure of LCA groups, it has a maximal compact subgroup $W$ and the quotient $H/W$ is isomorphic to $\R^k$ for some $k$. Then $H/W$ is a quotient module of $V$ and hence is distal as well; since it is contractable we deduce $H/W=\{0\}$, so $H=W$ is compact. Also, $W$ is compact and contractable, hence trivial. So $H=\{0\}$.
  
$\Hom(V,E)=0$. This is clear since $V$ is connected and amorphic implies totally disconnected.

$\Hom(C,V)=0$, $\Hom(C,E)=0$. This is clear since no $A$-orbit in $V$ or $E$ accumulates at zero (for $V$, this is Corollary \ref{ccogu}).

$\Hom(E,V)=0$, $\Hom(E,C)=0$. This is clear since $V$ and $C$ have no nonzero compact submodule, and $E$ is the union of its compact submodules.
\end{proof}

\subsection{Amorphic modules}

\begin{lem}\label{amorphicsub}
Every closed submodule or quotient of an amorphic module is amorphic.
\end{lem}
\begin{proof}
Since being amorphic is invariant under Pontryagin duality, it is enough to consider submodules. Namely, let $M$ be amorphic, $\Omega$ a compact open submodule, $N$ a closed submdule. Since $\Omega$ is profinite, so is $N\cap\Omega$, and the discrete quotient $N/(N\cap\Omega)$ being a submodule of the locally finite discrete module $M/\Omega$, it is locally finite as well. Hence $N$ is amorphic.
\end{proof}

\subsection{Facts on LCA groups}

For $M$ an LCA-group, we can view it as module over the trivial group. In particular, $\Ell(M)$ is just the union of all compact subgroups. When $M$ is an $A$-module, we can also view it as plain group, and we can thus write $\Ell_A(M)$ and $\Ell_{\{0\}}(M)$ to specify the module structure. We always have $\Ell_A(M)\subseteq\Ell_{\{0\}}(M)$. 

Similarly, with self-explanatory notation, we have $\Ell^\flat_{\{0\}}(M)\subseteq\Ell^\flat_A(M)$ and $\Omega^\sharp_A(M)\subseteq\Omega^\sharp_{\{0\}}(M)$. In particular, every sheer $A$-module is sheer as LCA group.


\begin{lem}\label{mctm1}
Let $M$ be an LCA-group. The natural homomorphism $M^\circ\times \Ell_{\{0\}}(M)\to M$ defined by $(x,y)\mapsto x-y$, is proper with open image. Its kernel is the diagonal of $\Ell(M^\circ)$. Its discrete cokernel is a torsion-free abelian group, and is the largest such quotient for $M$.
\end{lem}
\begin{proof}
Let us first check that it has an open image, i.e., that $M^\circ+\Ell(M)$ is open. This property is trivial when $M$ is connected. It is inherited from open subgroups, and from quotient by compact subgroups. Since $M$ is compact-by-Lie, it is therefore inherited by $M$. 

For properness, first assume that $M$ is $\sigma$-compact. The kernel is compact (namely the diagonal of $\Ell(M^\circ)$) and its image is open by the previous paragraph, hence is closed. So properness follows. In general, properness follows from being proper on every $\sigma$-compact open subgroup.

The second assertion (on the kernel) is straightforward.

Let us now prove the third assertion; we can suppose that $\Ell(M^\circ)=\{0\}$. First assume that $M/M^\circ$ is elliptic, in which case we have to show $M=M^\circ+\Ell(M)$. Modding out by a compact subgroup, we can suppose that $M$ is Lie. Then the divisible group $M^\circ$ has a complement in $M$ as abstract group. This complement is a discrete subgroup, hence is a complement, and is actually locally finite, thus equals $\Ell(M)$. 

Now, let us prove the second assertion in general. If by contradiction $M/M^\circ \Ell(M)$ is not torsion-free, it has a nontrivial finite subgroup. Let $N$ be its inverse in $M$. Then $N/N^\circ \Ell(N)$ is finite and nontrivial. So $N$ is connected-by-elliptic, implying, by the previous paragraph, $N=N^\circ \Ell(N)$, contradiction.

Finally, clearly every homomorphism to a torsion-free discrete group is trivial on both $M^\circ$ and $\Ell(M)$, whence the ``largest" assertion.
\end{proof}


\begin{cor}\label{semidec1}
Let $M$ be a sheer LCA-group (that is, $\Ell(M^\circ)=\{0\}$ and $M$ has no nonzero torsion-free discrete quotient). Then $M$ is the topological direct product of the Euclidean group $M^{\circ}$ and the totally disconnected group $\Ell_{\{0\}}(M)$. In addition, this is the only decomposition of $M$ as topological direct product of connected and totally disconnected LCA groups. In particular, whenever $M$ is also an $A$-module, this is also a decomposition as $A$-module.\qed
\end{cor}
\begin{proof}
This is a particular case of Lemma \ref{mctm1}: in this case the map $M^\circ\times \Ell_{\{0\}}(M)\to M$ of Lemma \ref{mctm1} is an isomorphism: injectivity holding because the kernel $\Ell(M^\circ)$ is trivial, surjectivity holding because $M$ is upper-sheer. If $M=C\times T$ is another decomposition ($C$ connected, $T$ totally disconnected), clearly $C=M^\circ$. Also $T$ being totally disconnected and upper-sheer, we have $T=\Ell_{\{0\}}(T)$, so $T\subseteq \Ell_{\{0\}}(M)$ and then $M=C+T$ ensures $T=\Ell_{\{0\}}(M)$.
\end{proof}

\subsection{Cocompact subgroups}

The following is classical, see for instance \cite[Prop.\ 4.C.11(2)]{CH}.

\begin{prop}\label{cgcoco}
Let $G$ be an LC group and $H$ a closed cocompact subgroup. Then $G$ is compactly generated if and only $H$ is compactly generated.\qed
\end{prop}

Although we only the next proposition in the abelian case, we write it, for reference, in more generality.

\begin{prop}\label{fgco}
Let $G$ be a compactly generated LC group. Then $G$ has a finitely generated (f.g.) subgroup with cocompact closure.
\end{prop}
\begin{proof}
1) If $G$ is connected and Lie, then it has a f.g.\ dense subgroup.

2) If $G$ is totally disconnected, let $X$ be a Cayley-Abels for $G$, that is, a nonempty connected graph of finite valency with a vertex-transitive action of $G$ with compact vertex stabilizers. It exists \cite[Prop.\ 2.E.9]{CH}. Fix a vertex $x_0$ of $X$ and let $S$ be a finite subset of $G$ such that the neighbors of $x_0$ are the $sx_0$ for $s\in S$. Let $\Gamma$ be the subgroup generated by $S$. Then $\Gamma$ acts transitively on $X$. Indeed, if $x\in\Gamma x_0$, writing $x=\gamma x_0$, each neighbor of $x$ has the form $\gamma s x_0$ for some $s\in S$, and hence also belongs to $\Gamma x_0$. If $K$ is the stabilizer of $x_0$, we then have $G=\Gamma K$, so $\overline{\Gamma}$ is cocompact. 

3) If $\Ell(G^\circ)$ is trivial, then $G^\circ$ is Lie and hence has a dense f.g.\ subgroup, and since $G/G^\circ$ has f.g.\ subgroup with cocompact closure, so does $G$.

4) In general, the previous case applies to $G/\Ell(G^\circ)$. Lifting a f.g.\ subgroup of the latter, we obtain a f.g.\ subgroup $\Gamma$ such that $\overline{\Gamma}\Ell(G^\circ)$ is cocompact. Thus $\overline{\Gamma}$ is cocompact.
\end{proof}

\section{One automorphism}

We recall that an automorphism $\phi$ of a locally compact group $G$ is {\bf contracting} if $\lim_{n\to\infty}\phi(x)=1$ for every $x\in G$. In this case, the convergence is uniform on compact subsets of $G$ \cite[Lemma 1.4(iv)]{Si}. The automorphism is said to be {\bf compacting} if there exists a {\bf compacting subset}, that is, a compact subset $\Omega\subseteq G$ with $\alpha(\Omega)\subseteq\Omega$ and $G=\bigcup_{n\ge 0}\alpha^{-n}(\Omega)$. Note that such $\Omega$ has nonempty interior (by Baire's theorem). If $M$ is totally disconnected, $\Omega$ can be chosen to be a compact open subgroup. If $\alpha$ is compacting with compacting subset $\Omega$, the intersection $\bigcup_{n\ge 0}\alpha^n(\Omega)$ is a compact subgroup, called limit subgroup of $\alpha$; it is independent of the choice of $\Omega$.

Observe that $\alpha$ is contracting if and only every compact neighborhood of $1$ is a compacting subset, if and only $\alpha$ is compacting with limit subgroup reduced to $\{1\}$.

For $\alpha\in A$, we say that $A$ is $\alpha$-contractive (resp.\ $\alpha$-compactive) if $\alpha$ acts on $A$ as a contracting (resp.\ compacting) automorphism. If $A$ is $\alpha$-compactive with limit subgroup $L$, then $L$ is a compact $A$-submodule and $M/L$ is $\alpha$-contractive.

We say that $M$ is {\bf $\alpha$-contractable} if there exists $\alpha\in A$ such that $M$ is $\alpha$-contractive.



\subsection{Willis' theory}\label{wilt1}

Let $G$ be a totally disconnected locally compact group and $\alpha$ a topological automorphism of $G$.

If $V$ is a compact open subgroup of $G$, define $V^\pm=\bigcap_{n\ge 0}\alpha^{\pm n}(V)$; these are compact subgroups of $V$; note that $\alpha^{\pm 1}(V^\pm)\supset V^\pm$.


\begin{thm}[Willis, Theorem 1 in \cite{Wil}]\label{t_tidy1}
For every $G$ and $\alpha$ there exists a compact open subgroup $V$ such that 
\begin{enumerate}
\item $V=V^+V^-=V^-V^+$;
\item the subgroups $V^{++}=\bigcup_{n\ge 0}\alpha^{n}(V^+)$ and $V^{--}=\bigcup_{n\ge 0}\alpha^{-n}(V^-)$ are closed subgroups of $G$. In particular, $V^{++}$ is $\alpha^{-1}$-compactive and $V^{--}$ is $\alpha$-compactive.
\end{enumerate}
\end{thm}

The theorem is stated only for conjugation automorphisms in \cite{Wil} but, as observed in subsequent papers, the general statement follows by considering the semidirect product $G\rtimes\langle\alpha\rangle$.


\subsection{Closedness of the set of contracted points}

Now and in the sequel, when we consider $\Z[\alpha^{\pm 1}]$-module, it is understood that the acting group is the infinite cyclic group $A=\langle\alpha\rangle$.

Let $M$ be a locally compact $\Z[\alpha^{\pm 1}]$-module. Denote by $\Con_\alpha(M)$ the set of elements $x$ in $M$ such that $\lim_{n\to +\infty}\alpha^n(x)=0$. This is a possibly non-closed submodule (see Remark \ref{nonclosed}).

\begin{prop}\label{uac}
If $M$ is lower-sheer as $\Z[\alpha^{\pm 1}]$-module (i.e. $\Ell^\sharp_{\Z[\alpha^{\pm 1}]}(M)=\{0\}$), then $\Con_\alpha(M)$ is closed, and is an $\alpha$-contractive module.
\end{prop}
\begin{proof}

We can assume that $\Con_\alpha(M)$ is dense and we have to show that $M$ is $\alpha$-contractive.

Let $V$ be a compact open subgroup, write $V^-=\bigcap_{n\ge 0}\alpha^{-n}(V)$ (so $\alpha(V^-)\subseteq V^-$) and $V^{--}=\bigcup_{n\ge 0}\alpha^{-n}(V^-)$. For each $x\in \Con_\alpha(M)$, the element $\alpha^n(x)$ belongs to $V$ for $n$ large enough, say $n\ge n_0$. Hence $\alpha^{n_0}(x)\in V^-$, and therefore $x\in V^{--}$. This proves that $\Con_\alpha(M)\subseteq V^{--}$.

Willis' theorem (Theorem \ref{t_tidy1}) ensures that there exists $V$ such that $V^{--}$ is closed; we assume $V$ has this property. Then the density of $\Con_\alpha(M)$ ensures that $V^{--}=M$. Note that by Baire's theorem, this implies that $V^-$ has nonempty interior, hence is open.

Define $F=\bigcap_{n\ge 0}\alpha^n(V^-)$. Then $L$ is a compact submodule and $\alpha$ induces a contracting automorphism of $M/L$. To conclude, it remains to show that $F=\{0\}$.
Since $\Ell^\flat(M)=0$, to show that $F=\{0\}$ it is enough to check that $F$ has no proper open submodule. Otherwise, modding out, we can suppose that $F$ is finite, and let us deduce that $F=\{0\}$.

Let $\Omega$ be a compact open subgroup of $M$ with $\Omega\cap F=\{0\}$. Since $\alpha$ is contracting modulo $F$ there exists a compact open subgroup $K$ of $\Omega$ with $\alpha(K)\subseteq K+F$, and $K$ can be chosen arbitrary small. In particular, $K$ can be chosen to be contained in $\alpha^{-1}(\Omega)$. So $\alpha(K)\subseteq (K+F)\cap \Omega=K$.
Write $F'=F\smallsetminus\{0\}$. Since $\alpha(F')=F'$, we deduce $\alpha(K+F')\subseteq K+F'$. It follows that $K+F'\cap\Con_\alpha(M)=\emptyset$. By density of $\Con_\alpha(M)$, we deduce that the open subset $K+F'$ is empty, and hence $F=\{0\}$.
\end{proof}

\begin{rem}\label{nonclosed}
Without lower-sheerness of $M$, closedness of $\Con_\alpha(M)$ can fail, even when $M$ is totally disconnected or connected. Let $B$ be a nontrivial finite abelian group, and $M=B^\Z$, endowed with the shift automorphism $f\mapsto (n\mapsto f(n+1))$. Then $\Con_\alpha(M)$ is the set of eventually (at $+\infty$) zero sequences, and thus is not closed. Also, the automorphism of the 2-torus $\R^2/\Z^2$ induced by the matrix $\begin{pmatrix}2&1\\1&1\end{pmatrix}$ has a non-closed contraction subgroup. These examples are classical.
\end{rem}

\begin{lem}\label{compactive1}
Let $M$ be a lower-sheer locally compact $\Z[\alpha^{\pm 1}]$-module. Suppose that $M$ is $\alpha$-compactive. Then $M$ is sheer, $M=\Con_\alpha(M)\times \Ell(M)$ and $\Ell(M)$ is compact.
\end{lem}
\begin{proof}
If $N$ is a discrete quotient of $M$, then $N$ is also $\alpha$-compactive, and clearly this implies that $N$ is a locally finite module. Thus $M$ is upper-sheer, and hence sheer. In particular, it is sheer as LCA-group.

By Theorem \ref{closed-pr}, $\Ell(M)$ is closed, and has a compact open submodule $\Omega$. Then $\Ell(M)/\Omega$ is $\alpha$-compactive and discrete, hence finite. Hence $\Ell(M)$ is compact.

Now let us prove the decomposition statement. By Corollary \ref{semidec1}, we can suppose that $M$ is either Euclidean or totally disconnected. The Euclidean case is clear: in this case, since $M$ has no nonzero compact submodule, it is $\alpha$-contractive.

Assume $M$ totally disconnected. By Proposition \ref{uac}, $\Con_\alpha(M)$ is closed.
Consider the natural homomorphism $f:\Con_\alpha(M)\times \Ell(M)\to M$. It is proper, by compactness of $\Ell(M)$. Its kernel is therefore compact, and hence its projection to $\Con_\alpha(M)$ is a compact $\alpha$-contractive module, hence trivial. Thus the kernel is contained in $\Ell(M)$ on which the homomorphism is clearly injective. Finally, surjectivity of $f$ is given by \cite[Corollary 3.17]{BW}.
\end{proof}


\subsection{A lemma on contraction minus identity}

For a $\Z[\alpha^{\pm 1}]$-module, recall that it is understood that $A=\langle\alpha\rangle$. To be polycontractable just means to be the direct sum of an $\alpha$-contractive and an $\alpha^{-1}$-contractive module.

\begin{lem}\label{difpro1}
Let $M$ be a polycontractable locally compact $\Z[\alpha^{\pm 1}]$-module. Then the multiplication by $(\alpha-1)$ is an automorphism of the topological module $M$.
\end{lem}
\begin{proof}
It is enough to assume that $M$ is either $\alpha$-contractive or $\alpha^{-1}$-contractive. If the $\alpha$-contractive case works, the other case follows since $\alpha^{-1}-1=-\alpha^{-1}(\alpha-1)$. So assume that $M$ is $\alpha$-contractive.

By Siebert's theorem \cite[Proposition 4.2]{Si} (or Corollary \ref{semidec1}), we can suppose that $M$ is either Euclidean or totally disconnected. The Euclidean case being straightforward from linear algebra, assume that $M$ is totally disconnected.

Let $\Omega$ be a compact open $\Z[\alpha]$-submodule. Observe that for every $v\in M$, we have
\begin{equation}\label{eqeq1}v\in\Omega\quad\Leftrightarrow\quad\alpha v-v\in\Omega.\end{equation} The implication $\Rightarrow$ is clear. Let by contradiction $v$ is a counterexample to $\Leftarrow$, namely $\alpha v-v\in\Omega$ and $v\notin\Omega$. Considering $n$ maximal such that $\alpha^nv\notin \Omega$ (which exists by contractivity), and define $w=\alpha^nv$. Then $w\notin\Omega$ and $\alpha w\in\Omega$. Since $\alpha v-v\in\Omega$ and $\alpha\Omega\subseteq\Omega$, we have $\alpha^n(\alpha v-v)=\alpha w-w\in\Omega$. Hence $w\in\Omega$, a contradiction.


Consider the ultrametric on $M$ defined by $d(x,y)=\exp(\inf\{n:\alpha^n(x-y)\in\Omega\})$. It is complete, addition-invariant and defines the topology of $M$. Moreover, by (\ref{eqeq1}), the multiplication by $(\alpha-1)$ is an isometric embedding of $(M,d)$ into itself; in particular its image is closed, by completeness. It remains to check that it is surjective: indeed, if we mod out by the image, we obtain a module on which $\alpha$ is contracting but acts as the identity, so we obtain the trivial module.
\end{proof}

\subsection{Decomposition under a single automorphism}

\begin{lem}\label{singledec}
Let $M$ be a sheer LC $\Z[\alpha^{\pm 1}]$-module. Then $M$ decomposes as a topological $A$-module:
\[M=\Con_\alpha(M)\oplus\Con_{\alpha^{-1}}(M)\oplus \Ell(M)\oplus \Dist(M),\]
where $\Dist(M)$ is the largest distal subspace of the Euclidean $A$-module $M^\circ$.
\end{lem}
\begin{proof}
For a direct product of an Euclidean and a totally disconnected module, all given subspaces split accordingly. Hence Corollary \ref{semidec1} reduces to the cases when $M$ is either Euclidean or totally disconnected (to apply the corollary, note that $M$ being a sheer $A$-module, it is a sheer LCA group).

The case when $M$ is Euclidean is just plain linear algebra: the subspaces in the decomposition are then as follows. The sum of characteristic subspaces of $\alpha$ with respect to complex eigenvalues of modulus $>1$ (resp.\ $<1$, resp.\ $=1$) is $\Con_\alpha(M)$ (resp.\ $\Con_{\alpha^{-1}}(M)$, resp.\ $\Dist(M)$), and $\Ell(M)=\{0\}$.

Now assume that $M$ is totally disconnected. We then have to prove that $M=\Con_\alpha(M)\oplus\Con_{\alpha^{-1}}(M)\oplus \Ell(M)$.
By Willis' theorem (Theorem \ref{t_tidy1}), there are closed submodules $N_+$, $N_-$ such that $N_+$ is $\alpha$-compactive, $N_-$ is $\alpha^{-1}$-compactive, and $N_++N_-$ is open. Therefore, by Lemma \ref{compactive1}, there are closed submodules $C_+$, $C_-$, and a compact submodule $\Omega$ such that $C_+$ is $\alpha$-contractive, $C_-$ is $\alpha^{-1}$-contractive, $\Omega$ is compact, and $C_++C_-+\Omega$ is open. Note that $C_\pm=\Con_\pm(M)$, and $\Omega\subseteq\Ell(M)$.

Consider the sum homomorphism $f:\Con_\alpha(M)\oplus\Con_{\alpha^{-1}}(M)\oplus\Ell(M)\to M$. Let $K$ be its kernel.
Let $s=(c_+,c_-,z)$ be an element of $K$, in this decomposition. Then $c_++c_-+z=0$ in $M$. Thus, for every $n$, $0=\alpha^n(c_+)+\alpha^n(c_-)+\alpha^n(z)$. Then, for $n\to\infty$, $\alpha^n(c_++z)$ is bounded, while $\alpha^n(c_-)$ tends to infinity unless $c_-=0$. This forces $c_-=0$. Similarly, letting $n$ tend to $-\infty$ ensures $c_+=0$. Hence $c_+=c_-=0$, so $z=-c_+-c_-=0$ showing that $K=\{0\}$, that is, $f$ is injective. All this argument showing that $f$ is injective with open image is true for some $\sigma$-compact submodule of $M$. Hence $f$ is proper.


It remains to see that $f$ is surjective. Write $T=\Con_\alpha(M)+\Con_{\alpha^{-1}}(M)+\Ell(M)$, which is the image of $f$ and which we already proved to be open; let us show that $M=T$. Since $M$ is sheer, the discrete quotient $M/T$ is locally finite as $A$-module.

Take $x\in M$. Since the discrete module $M/T$ is a locally finite module (because $M$ is sheer), for some $k\ge 1$ we have $\alpha^kx=x$ in $M/T$, i.e.\ $(\alpha^k-1)x\in T$. Write $P=\Con_\alpha(M)\oplus\Con_{\alpha^{-1}}(M)$, so that $T=P\oplus\Ell(M)$.
We can decompose $(\alpha^k-1)x\in T$ according to this decomposition. By Lemma \ref{difpro1}, we can write $(\alpha^k-1)x=(\alpha^k-1)u+v$ with $u\in P$ and $v\in \Ell(T)$. Set $y=x-u$, so that $(\alpha^k-1)y=v$. Modulo the closed submodule $N$ generated by $v$ (which is compact), we have $\alpha^ky=y$, and therefore, modulo $N$, the submodule generated by $y$ is actually generated, as a group, by the finite set $\{y,\dots,\alpha^{k-1}y\}$. Since $M$ is elliptic as an LCA-group, we deduce that the closed submodule of $M$ generated by $y$ is compact, and therefore $y\in \Ell(M)$. So $x=u+y\in P+\Ell(M)\subseteq T$.
\end{proof}


\section{Structure of sheer modules}

\subsection{Decomposition of sheer modules}\label{desheer}
\begin{thm}\label{sheerdec1}
Suppose that $A$ is compactly generated. Every sheer module has a unique decomposition $P\oplus E\oplus V$, with $P$ polycontractable, $E$ amorphic, $V$ distal Euclidean.
\end{thm}

We need to keep track of the construction in the proof. Therefore we introduce some definition. For a sheer $A$-module $M$ and $\alpha\in A$, define $\Con^{\le 0}_\alpha(M)=C_{\alpha^{-1}}(M)\oplus \Ell_{\Z[\alpha^{\pm 1}]}(M)\oplus \Dist_{\Z[\alpha^{\pm 1}]}(M)$. So Lemma \ref{singledec} says that $M=\Con_\alpha(M)\oplus\Con^{\le 0}_\alpha(M)$. If $M$ is a sheer $A$-module and $\alpha\in A$, this is an $A$-module decomposition.

Now let $(\alpha_i)_{i\ge 1}$ be a sequence in $A$. For a sheer LC $A$-module $M$ and define $B_0=M$, and by induction $C_i=\Con_{\alpha_i}(B_{i-1})$ and $B_i=\Con^{\le 0}_{\alpha_i}(B_{i-1})$, so that $B_{i-1}=C_i\oplus B_i$. Then for each $n$, we have $M=C_1\oplus\dots \oplus C_n\oplus B_n$.

\begin{lem}\label{lemdecompo}
Suppose that for some $n$, we have $S=\{\alpha_i:i\le n\}$ satisfies the following: $S=S^{-1}$, and $\overline{\langle S\rangle}$ is cocompact in $A$ (this exists as soon as $A$ is compactly generated).
 Then $\Ell(B_n)$ is an amorphic $A$-module and $\Dist(B_n)=B_n^\circ$ is a distal Euclidean $A$-module. 
\end{lem}
\begin{proof}
Using Corollary \ref{semidec1}, we can assume that $M$ is either totally disconnected or Euclidean (since the decomposition is preserved at each step).

We start with the easier Euclidean case: we have to prove that $B_n$ is distal. Complexify and triangulate the $A$-action on the real vector space $B_n$. Consider a diagonal coefficient, given by a continuous homomorphism $\omega:A\to\C^*$. Then by construction, for each $\alpha\in S$ we have $|\omega(\alpha)|\le 1$. Since $S=S^{-1}$, we deduce that $S\subseteq\Ker(|\omega|)$. Hence the homomorphism $\log|\omega|:A\to \R$ factors through the compact group $A/ \overline{\langle S\rangle}$, and hence it is zero. Thus $|\omega|=1$, that is, $B_n$ is distal.

Now suppose that $M$ is totally disconnected. For every $\alpha\in S$, the module $B_n$ is amorphic over $\Z[\alpha^{\pm 1}]$. For $T\subseteq S$ symmetric, we prove by induction on $|T|$ that there is a compact open $\langle T\rangle$-submodule; the case $T$ empty is clear. More precisely, suppose that $T$ has maximal cardinal such that this holds, with a compact open $\langle T\rangle$-submodule $\Omega$, and by contradiction choose $\alpha\in S\smallsetminus T$. Since $B_n$ is amorphic over $\Z[\alpha^{\pm 1}]$, the union $\bigcup_{n\in\Z}\alpha^nT$ has compact closure. Hence the closure of the additive subgroup by this union is compact and invariant under $T\cup\{s,-s\}$, contradiction. 

Thus $M$ is an amorphic $\langle S\rangle$-module. Let $\Omega$ be a compact open $\langle S\rangle$-submodule. Its stabilizer in $A$ is open and contains the subgroup $\langle S\rangle$ with cocompact closure, and hence is a finite index subgroup $B$ of $A$. Therefore the sum $\sum_{a\in A}a\Omega$ is a finite sum, and thus is a compact open $A$-submodule.

The existence of $(\alpha_i)$ is ensured by Proposition \ref{fgco}: find a finite family $(\alpha_1,\dots,\alpha_m)$ generating a dense subgroup of a cocompact subgroup. Then set $n=2m$, $\alpha_{m+i}=\alpha_i$ for $1\le i\le m$, and $\alpha_i=1$ for $i>2m$.
\end{proof}


\begin{proof}[Proof of Theorem \ref{sheerdec1}]
Lemma \ref{homzero} ensures uniqueness: if $P_1\oplus E_1\oplus V_1$ is another decomposition, considering projections $P_1\to E$, $P_1\to V$, the lemma then implies these are zero, so $P_1\subseteq P$ and similarly $E_1\subseteq E$, $V_1\subseteq V$, whence equality holds. 

For the existence, apply Lemma \ref{lemdecompo} (for a suitable sequence $(\alpha_i)$, which exists since $A$ is compactly generated). Then, with the notation preceding the lemma, $M=C_1\oplus\dots \oplus C_n\oplus B_n$ with each $C_i$ contractable, and the lemma says that $B_n$ is the direct sum of an amorphic and a distal Euclidean module.
\end{proof}

\begin{cor}\label{maxcp}
Let $A$ be compactly generated. Then for every compactly generated LC $A$-module $M$, $\Ell(M)$ is compact. That is, $M$ has a maximal compact submodule.
\end{cor}
\begin{proof}
Modding out by the compact submodule $\Ell^\flat(M)$, we can suppose that $M$ is lower-sheer. By Theorem \ref{sheerdec1}, write $N=\Omega^\sharp(M)=P\oplus E\oplus V$ with $E=\Ell(N)$, $P$ polycontractable, $V$ distal Euclidean. Let $E'$ be a compact open submodule of $E$. Then $Q=M/(P\oplus E'\oplus V)$ is a discrete $A$-module, with the locally finite submodule $E'/E$.

Let $B$ be a cocompact lattice in $A$. Then since $Q\rtimes A$ is a compactly generated LC group, so is its cocompact lattice $Q\rtimes B$ (see Proposition \ref{cgcoco}). Hence $Q$ is a finitely generated $B$-module. By noetherianity, we deduce that $E/E'$ is also a finitely generated $B$-module, since it is also locally finite as $B$-module, it is finite. Thus $E$ is compact.

Since the projection of $\Ell(M)$ in the purely discrete group $M/N$ is trivial, we deduce that $\Ell(M)=\Ell(N)=E$ is compact.
\end{proof}


\subsection{Sheer submodules and quotients}

\begin{thm}\label{sheersub}
Suppose that $A$ is compactly generated. Let $M$ be a sheer $A$-module and $N$ a closed submodule. Then $N$ is sheer if and only if $M/N$ is sheer.
\end{thm}

As far as we tried, this is not a plain consequence of the definition. The proof uses Theorem \ref{sheerdec1}, and, more specifically, its more explicit form given in Lemma \ref{lemdecompo}, which outputs a decomposition with some specific properties for the module and its submodule simultaneously. Indeed, one main issue is that a quotient of a polycontractable module need not be polycontractable. However, this is true for a quotient of a contractable module, and we need to be able to exploit this.

\begin{proof}
Because of Pontryagin duality, it is enough to prove one direction. Suppose that $N$ is sheer. 

We apply Lemma \ref{lemdecompo} simulatenously to $M$ and $N$. Namely $M=C_1\oplus\dots \oplus C_n\oplus B_n$, $B_n=\Ell(B_n)\oplus B_n^\circ$ with $\Ell(B_n)=\Ell(M)$ amorphic and $B_n^\circ$ distal Euclidean, and $C_i=\Con_{\alpha_i}(C_i\oplus \dots\oplus C_n\oplus B_n)$. Consider the same decomposition for $N$, namely $N=C'_1\oplus\dots \oplus C'_n\oplus B'_n$. Then from the definition it follows that $C'_i=C_i\cap N$. Also $\Ell(B'_n)\subseteq\Ell(M)=B_n$, and $(B'_n)^\circ\subseteq B_n^\circ$. So $M/N=C''_1\oplus\dots\oplus C''_n\oplus B''_n$ with $C''_i=C_i/C'_i$ contracted by $\alpha_i$ and $B''_n=B_n/B'_n$. Clearly $B_n^\circ/(B'_n)^\circ$ is distal Euclidean, and $\Ell(B_n)/\Ell(B'_n)$ being quotient of the amorphic module $\Ell(B_n)$, has to be amorphic as well (Lemma \ref{amorphicsub}). Hence $M/N$ is sheer.
\end{proof}

\subsection{Sheer core and envelope}

\begin{thm}\label{corenv}
Suppose that $A$ is compactly generated. Let $M$ be a fixed sheer LC $A$-module. For a submodule $N$ of $M$, write $N^\boxminus=\Omega^\sharp(N)$  ({\bf sheer core of $N$}) and let $N^\boxplus$ ({\bf sheer envelope of $N$ in $M$}) be the inverse image of $\Ell^\flat(M/N)\subseteq M/N$ in $M$. Then $N^\boxminus\subseteq N\subseteq N^\boxplus$, $N/N^\boxminus$ is purely discrete, and $N^\boxplus/N$ is compact parafinite. 
The submodule $N^\boxminus$ is sheer and maximal for this property among submodules of $N$; the submodule $N^\boxplus$ is sheer and minimal for this property among submodules containing $N$. If $N_1\subseteq N_2$ then $N_1^\boxminus\subseteq N_2^\boxminus$ and  $N_1^\boxplus\subseteq N_2^\boxplus$.
\end{thm}

\begin{rem}
While $N\mapsto N^\boxminus$ is intrinsic to $N$, the assignment $N\mapsto N^\boxplus$ definitely depends on $M$. This may sound strange in view of Pontryagin duality, but rather, the dual statement is that the assignment $M/N\mapsto M/N^\boxplus$ only depends on $M/N$.
\end{rem}

\begin{proof}[Proof of Theorem \ref{corenv}]
It is enough to prove all assertions about the sheer core, since those about the sheer envelope follow by Pontryagin duality (using Theorem \ref{sheersub} in some cases). 

The assertion $N^\boxminus\subseteq N$ is trivial. The $N^\boxminus$ is sheer: it is lower-sheer as submodule of the sheer module $M$, and is upper-sheer by Proposition \ref{omush}. That $N/N^\boxminus$ is purely discrete is part of Theorem \ref{cafi}. If $N_1\subseteq N_2$, then by Proposition \ref{omush} applied successively in $N_1$ and $N_2$, the submodule $N_1^\boxminus$ is upper-sheer, and hence contained in $N_2^\boxminus$.
\end{proof}

\subsection{Chains of submodules}

We start with the following lemma. Exceptionally, we write it in the non-commutative setting, since the proof is not longer than the case of modules.

\begin{lem}\label{boundchainc}
Let $M$ be a contractable LC $A$-module. Then there is a bound (depending only on $M$) on the length of chains of closed submodules of $M$. More generally, if $G$ is an LC-group with a contracting automorphism $\alpha$, there is a bound on the length of chains of $\alpha$-invariant subgroups ($H$ $\alpha$-invariant meaning $\alpha(H)=H$).
\end{lem}
\begin{proof}
First case: $G$ is totally disconnected. This case is actually covered by \cite[Theorem 3.3]{GW}; let us provide a much shorter argument. Let $\alpha^{-1}$ multiply the Haar measure of $G$ by $n_G$. Then $n_M$ is a positive integer, namely it equals the index $[\Omega:\alpha(\Omega)]$ for some/every compact open subgroup $\Omega$ of $G$ such that $\alpha\Omega\subseteq\Omega$. Now let us check that if $H$ is an $\alpha$-invariant closed subgroup then $n_H\le n_G$, and if $n_H=n_G$ then $H=G$. Then $n_G=[G:\alpha(G)]$ and $n_H=[\Omega\cap H:\alpha(\Omega\cap H)]$. The inequality immediately follows; if it is an equality we deduce that $\Omega=\alpha(\Omega)(\Omega\cap H)$. It follows that the image of $\Omega$ in $G/H$ is a compact $\alpha$-invariant subset. Since $\alpha$ acts as a contraction (to the base-point) on the coset space $G/H$, we deduce that this is a point, i.e.\ $\Omega\subseteq H$. Thus $H$ is open, and since the discrete coset space $G/H$ admits a contraction, we deduce that it is reduced to a point and thus $H=G$. The bound immediately follows (e.g., $\log_2(n_G)$ is such a bound).

Second case: $G$ is connected, namely a $d$-dimensional simply connected nilpotent Lie group. Then since every contractable subspace is a real subspace, we have the bound $d$.

The general case then follows from the Siebert's decomposition \cite{Si}.
\end{proof}

\begin{cor}\label{chain-poly}
Let $M$ be a polycontractable LC $A$-module. Then there is a bound (depending only on $M$) on the length of chains of {\bf sheer} closed submodules of $M$.
\end{cor}
\begin{proof}
Given Lemma \ref{lemdecompo}, this follows from the contractable case, namely Lemma \ref{boundchainc}.
\end{proof}

\begin{rem}
Corollary \ref{chain-poly} does not extend to general closed submodules. For instance, we have the bi-infinite sequence of lattices $3^n\Z[1/2]$ ($n\in\Z$) in $\R\times\Q_2$ (viewed as $\Z[\alpha^{\pm 1}]$-module, $\alpha$ acting by multiplication by $2$).
\end{rem}




An LC $A$-module is {\bf topologically characteristically simple} if $M\neq \{0\}$ and the only closed $\Aut(M)$-invariant closed submodules are $\{0\}$ and $M$ ($\Aut(M)$ being the group of topological $A$-module automorphisms). It is {\bf topologically simple} if $M\neq\{0\}$ and its only closed submodules are $\{0\}$ and $M$. Of course topologically simple implies topologically characteristically simple.

\begin{thm}
Let $M$ be a topologically characteristically simple LC $A$-module. Then exactly one of the following holds:
\begin{enumerate}
\item\label{ca1} $M$ is finite;
\item\label{ca2} $M$ is compact infinite;
\item\label{ca3} $M$ is discrete infinite; 
\item\label{ca4} $M$ is Euclidean and contractable;
\item\label{ca5} $M$ is Euclidean distal;
\item\label{ca6} $M$ is contractable and totally disconnected;
\item\label{ca7} $M$ is amorphic, neither discrete nor compact. In this case, $M$ is not topologically simple.
\end{enumerate}
\end{thm}
\begin{proof}
By assumption, $M\neq\{0\}$. Clearly, $M$ is either connected or totally disconnected.

If $M$ is connected, let us conclude (we will not need the assumption on $A$). Then $\Ell(M)$ is a compact characteristic subgroup, hence either is $M$ or trivial. In the first case, $M$ is compact (hence Case (\ref{ca2})), and otherwise, $M$ is Euclidean. If $M$ is not distal (Case (\ref{ca5})), let us check that $M$ is contractable (Case (\ref{ca4})). Indeed, this means that there is some $\alpha$ acting with an eigenvalue $z$ of modulus $\neq 1$. If $z$ is real (resp.\ non-real), the $\Ker(\alpha-z)$ (resp. $\Ker((\alpha-z)(\alpha-\bar{z}))$ a nontrivial subspace, invariant under $\Aut(M)$, hence equals $M$ which is then contractable.

Now suppose that $M$ is totally disconnected. Then $\Omega^\flat(M)$ is either zero or $M$, and in the first case, Theorem \ref{cafi} (which uses that $A$ is compactly generated) ensures that $M$ is purely discrete (Case (\ref{ca3})). Similarly $\Ell^\flat(M)$ is either zero or $M$, and in the second case, $M$ is compact paracompact (Case (\ref{ca2})). Otherwise, $M$ is sheer. Using that $A$ is compactly generated (and that $M$ is totally disconnected), Theorem \ref{sheerdec1} then ensures that $M$ is either contractable (Case (\ref{ca6})) or amorphic.

If $M$ is amorphic, then it is either compact (Case (\ref{ca2})) discrete (Case (\ref{ca3})), or none (Case (\ref{ca7})). In the latter case, it has a proper open submodule that is infinite, so is not topologically simple.
\end{proof}

\begin{rem}
By a simple commutative algebra exercise, if $M$ is discrete (Cases (\ref{ca1}) or (\ref{ca2})), the kernel of $A\to\Aut(M)$ is an open subgroup $B$, and there is a prime ideal $P$ of $\Z[A/B]$ (with $A/B$ acting faithfully on $\Z[A/B]/P$, i.e., $A/B\cap 1+P=\{1\}$) such that $M$ is a nonzero vector space over the field $\mathrm{Frac}(\Z[A/B])$.

Case (\ref{ca7}) is not much harder to describe: for some prime $p$, $M$ is a $\Z_p$-module and either $M$ is a finite-dimensional vector space over $\Q_p$, or, for some infinite set $I$, $M$ has an open subgroup isomorphic to $\Z_p^I$ and $pM$ is dense in $M$.

If $M$ is Euclidean, one can check that $M$ is a real or complex vector space with a scalar action, which, in the complex case, is given by a continuous homomorphism $A\to\C^*$ whose image is not contained in $\R^*$.
\end{rem}

\begin{thm}\label{ascending}
Suppose that $A$ is compactly generated. Let $M$ be a compactly generated $A$-module. Then there is a compact submodule $W$ of $M$, namely $W=\Ell(M)$, such that for every ascending sequence $(M_n)$ of closed submodules of $M$, for large enough $n$, $M_n$ is cocompact in $\overline{\bigcup_i M_i}$.
\end{thm}
\begin{proof}
By Corollary \ref{maxcp}, $\Ell(M)$ is compact. Hence, for the first assertion, there is no restriction in assuming that $\Ell(M)$ is trivial.

The module $\Omega^\sharp(M)$ is sheer. Let us use the notation $\boxplus$ relative to the ambient module $\Omega^\sharp(M)$. Since $\Ell(M)=\{0\}$, we can write $\Omega^\sharp(M)=P\oplus V$ with $P$ polycontractable and $V$ distal Euclidean. Hence, by Corollary \ref{chain-poly}, the sequence $(M_n\cap \Omega^\sharp(M))^\boxplus$ is stationary, say, equal to a certain submodule $L$. There is no restriction in assuming that $(M_n\cap \Omega^\sharp(M))^\boxplus=L$ for all $n$. By Theorem \ref{corenv}, $(M_n\cap \Omega^\sharp(M))^\boxplus$ contains $M_n\cap \Omega^\sharp(M)$ as closed cocompact subgroup. Hence $L\cap M_n$ is closed cocompact in $L$. It follows that $M'_n=L+M_n$ is closed and contains $M_n$ as closed cocompact subgroup.

We claim that $(M'_n)$ is stationary. We first observe that $M'_n\cap \Omega^\sharp(M)=L$. The inclusion $\subseteq$ is trivial. If $x\in M'_n\cap \Omega^\sharp(M)$, write it as $y+z$, $y\in L$, $z\in M_n$. Since $x\in\Omega^\sharp(M)$, we have $z\in \Omega^\sharp(M)$, so $z\in M_n\cap \Omega^\sharp(M)\subseteq L$ and hence $x\in L$. The observation ensures that $M'_n\cap\Omega^\sharp(M)$ is stationary. Then the claim follows because the projection of $M'_n$ in the discrete quotient is ascending, hence stationary by noetherianity.

Write $M'_n=M'$ for large $n$. Then $M_n$ is cocompact in $M'$ for large enough $n$. The result follows. 
\end{proof}

\begin{cor}\label{scgcg}
Let $A$ be compactly generated. Then every closed submodule of a compactly generated LC $A$-module is compactly generated.
\end{cor}
\begin{proof}
Let $M$ be a compactly generated LC module. Let $N$ be a closed submodule. Let $N_0$ be an open, compactly generated closed submodule of $N$. Since $M$ is $\sigma$-compact, so is $N$, and hence the discrete quotient $N/N_0$ is countable. Hence there is an ascending sequence $(N_n)$ of compactly generated submodules of $N$ with $N_0$ as previously defined, and $\bigcup N_i=N$. By Theorem \ref{ascending}, for some $n$, $N_n$ is cocompact in $N$. Since $N_n$ is compactly generated and cocompact in $N$, it follows that $N$ is compactly generated (Proposition \ref{cgcoco}).
\end{proof}

\section{Fine structure of contractable modules}

\begin{lem}\label{homlc1}
Let $M,M'$ be LC $A$-modules. Suppose $M\oplus M'$ is contractable.
Then $\Hom_{A}(M,M')$ (the group of continuous homomorphisms) is locally compact for the compact-open topology.
\end{lem}
\begin{proof}
The assumption means that there exists $\alpha\in A$ acting as a contraction on both $M$ and $M'$, we fix such an $\alpha$. Since $\Hom_{A}(M,M')$ is a topological group, it is enough to find a compact neighborhood of 0.

The zero homomorphism has a basis of neighborhoods of the form $V(K,K')$, where $K$ is a compact symmetric subset of $M$, $K'$ is a compact neighborhood of 0 in $M'$, where $V(K,K')$ is the set of continuous homomorphisms $M\to M'$ mapping $K$ into $K'$.

We first observe the following: if $K$ generates $M$ as $A$-module, then $V(K,K')$ is compact has a compact closure $V'$ in ${M'}^M$ (the set of maps $M\to M'$ endowed with the pointwise convergence topology, ignoring the topology on $M$). And moreover the closure of $V(K,K')$ consists of maps that are locally bounded (the image of every compact subset has compact closure).
Indeed if $P$ is a compact subset of $M$, then $P\subseteq \bigcup_{i=1}^n\alpha_iK$ for some finite family $(\alpha_i)_{1\le i\le n}$, and hence for every $f\in V(K,K')$ we have $f(P)\subseteq \bigcup_{i=1}^n\alpha_iK'$. (The existence of $K$ only makes use of the fact that $M$ is compactly generated, which follows from being contractable.)

Next, we show that every element $f$ of the closure $V'$ is continuous. Namely, we show that, assuming the existence of $\alpha$ as above, every locally bounded homomorphism $f$ is continuous. Indeed, if $(x_j)$ is a net tending to 0 in $M$, we can write $x_j=\alpha^{k_j}(y_j)$ with $(y_j)$ bounded and $k_j\to +\infty$. So $f(x_j)=\alpha^{k_j}f(y_j)\in \alpha^{k_j}(f(K))$. Since $f(y_j)$ is bounded, while $\alpha^{k_j}$ converges uniformly on bounded subsets to the constant $1$. Hence $f(x_j)$ tends to $0$.


Thus $V(K,K')$ is compact in the pointwise convergence topology. This implies it is also compact in the compact-open topology, by a general argument only based on the Baire theorem, see \cite{CGl}.
\end{proof}

\begin{thm}\label{afl1}
Let $M$ be a contractable $A$-module. Let $R$ be a closed $A$-subalgebra of $\End_{A}(M)$. Then every $R$-submodule of $M$ is closed and in particular $M$ has finite length over $R$. The underlying ring $R$ is an artinian ring (i.e., $R$ has finite length as $R$-module).
\end{thm}
\begin{proof}
Let us first show the closedness statement. Here it is enough to suppose that $R$ is the closure $R_0$ of the image of $\Z[A]$ and in particular is commutative.

Let $M$ be a counterexample, of minimal length as a topological module, and $N$ a non-closed submodule. By minimality, $N$ is dense, and does not contain any nonzero closed submodule of $M$.

If $x\in N\smallsetminus\{0\}$ then $N\cap\overline{R_0x}$ is a non-closed submodule of $M$ and hence is dense; in particular $\overline{R_0x}=M$ for every $x\in N$. 

If $M$ is not simple as a topological module, let $L$ be a closed submodule of $M$ distinct from $\{0\},M$. Then by minimality $N\cap L=\{0\}$ and the projection of $N$ on $M/L$ is all of $M/N$; in particular, $M=N\oplus L$ as an abstract module. The latter holds for every nonzero proper submodule $N$ and non-closed submodule $L$. This shows that there is no proper inclusion
\begin{itemize}\item between any two non-closed submodules of $M$;
\item between any two nonzero proper closed submodules of $M$. In particular, $M$ has length $2$ as a topological module.
\end{itemize}

It follows that $M$ also has length $2$ as an abstract module over $R_0$. Since it is a faithful $R_0$-module, we deduce that $R_0$ is artinian. More precisely, the existence of a faithful module that is the direct sum of two simple modules implies that either $R_0$ is a field, or the product of two fields. If by contradiction $R$ is the product of two fields, let $\pi$ and $1-\pi$ be its nontrivial idempotents. Then in the (then unique) direct product decomposition of $M$, the summands are $\Ker(\pi)$ and $\Ker(1-\pi)$, which are closed, a contradiction. Thus $R_0$ is a field.
So $R_0$ is a locally compact field and $M$ is isomorphic to $R_0^2$; in this case every $R_0$-submodule of $R_0^2$ is closed and we also have a contradiction with the assumption that $M$ is not topologically simple. Hence $M$ is simple as a topological module.


Observe that the multiplication by $\alpha$ in $\End_{A}(M)$, and hence in $R_0$, is contracting. Let us show that for every $m\in M$, the multiplication map $\End(M)\to M$ mapping $r$ to $rm$, is a proper map.

Indeed, let $(r_i)$ be a net tending to infinity in $\End(M)$. If $\Omega$ is a compact neighborhood of $0$, we can write $r_i=\alpha^{-k_i}(\phi_i)$, where $\phi_i\in\Omega\smallsetminus\alpha(\Omega)$, and $k_i$ tends to $+\infty$. By contradiction, we can assume that $r_im$ is bounded, and actually tends to an element $m'$ of $M$, and also that $\phi_i$ tends to an element $\phi$ of $\End(M)$. Define $e_i=r_im-m'$, which tends to 0. So $\phi_im-\alpha^{k_i}m'=\alpha^{k_i}e_i$, and hence $\phi_im$ tends to zero. So $\phi m=0$. Since the set of $m$ such that $\phi m=0$ is a closed submodule, we deduce that $\phi=0$, a contradiction. So the multiplication map $\End(M)\to M$ mapping $r$ to $rm$, is a proper map.

Fix a nonzero $m\in N$. We deduce that $R_0m$ is closed for every $m\in N$. Since we proved earlier that $R_0m$ is dense in $M$, we have $M=R_0m\subseteq N$, and thus $N=M$, a contradiction. Therefore the closedness assertion is proved. The second one then follows from Lemma \ref{boundchainc}.


(We do no longer assume $R_0=R$.) Since $R_0$ is commutative and $M$ is a faithful $R_0$-module of finite length, it follows that $R_0$ is an artinian ring. Now $R$ is also a contractable module and applying this to $M=R$, we obtain that $R$ is artinian (i.e.\ has finite length as left module over itself), since it has finite length over $R_0$, both as a left or right module.
\end{proof}

\begin{rem}
For $A=\langle\alpha\rangle$ infinite cyclic, in the case of the polycontractable module $M=\Z/p\Z\lp t\rp\times\Z/p\Z\lp t^{-1}\rp$ (with $\alpha$ acting by multiplication by $t$), the closure of the image of $\Z[\alpha^{\pm 1}]$ in $\End(M)$ is reduced to $\Z/p\Z[t^{\pm 1}]$. In particular, $M$ is not finitely generated over this closure.
\end{rem}

Let $\K$ be a locally compact field. Let $i:A\to\K^*$ be a continuous homomorphism whose image is not contained in the 1-sphere (the maximal compact subgroup of $\K^*$, which is the group of norm-1 element for any compatible multiplicative norm), such that the closed additive subgroup generated by $i(A)$ equals $\K$. The following corollary was obtained by Gl\"ockner and Willis for $A=\Z$ (the case of arbitrary $A$ is not more difficult, although we use here a different language even when $A=\Z$).


\begin{cor}\label{sico}
Each such module is simple contractable, and conversely every simple contractable $A$-module has this form.
\end{cor}
\begin{proof}
Given such data, if $W\subseteq\K$ is a nonzero closed submodule and $w\in W\smallsetminus\{0\}$, then $w^{-1}W$ is also a nonzero submodule, hence contains $i(A)$, and hence equals $\K$ by assumption. Hence $W=\K$ and $\K$ is a simple $A$-module. Since some element of $i(A)$ has norm $<1$, it acts as a contraction on $\K$.

Conversely, let $M$ be a simple contractable module. Then $\K=\End_A(M)$ is a locally compact ring (Lemma \ref{homlc1}), and is also, by Theorem \ref{afl1} an artinian ring with a faithful simple module, and hence is a field. Let $i$ be the canonical map $A\to\K$. Fixing any nonzero $m\in M$, the map $\K\to M$ mapping $r\mapsto rm$ yields the desired isomorphism.
\end{proof}

\begin{prop}
Let $M,S$ be nonzero contractable $A$-modules, with $S$ simple. Then $S$ is isomorphic to a quotient of $M$ if and only if it is isomorphic to a submodule of $M$, if and only if it is isomorphic to a subquotient of $M$.
\end{prop}
\begin{proof}
This follows from the analogous statement for finitely generated modules over artinian rings.
\end{proof}

\begin{defn}\label{d_monotypic}
The {\bf support} of a polycontractable $A$-module is the set of its isomorphic classes simple contractable subquotients (or equivalently, simple contractable submodules, or quotients).

The polycontractable module is {\bf monotypic} if its support is a singleton; the unique simple module in this singleton is called its {\bf type}. Two polycontractable modules are {\bf disjoint} if their supports are disjoint.
\end{defn}

\begin{prop}
Two polycontractable modules $M,M'$ are disjoint if and only if $\mathrm{Hom}_A(M,M')=\{0\}$.\qed
\end{prop}

\begin{prop}\label{monotypic}
Every polycontractable $A$-module $M$ has a unique decomposition $\bigoplus_i M_i$ as finite direct product of pairwise disjoint monotypic contractable submodules. Every polycontractable submodule $N$ decomposes accordingly, i.e., $N=\bigoplus (N\cap M_i)$.
\end{prop}
\begin{proof}
First suppose that $A$ is compactly generated. Lemma \ref{lemdecompo} reduces to when $M$ is contractable. Then the result follows from Theorem \ref{afl1} and the analogous result for finitely generated modules over artinian rings.

In general, there is a compactly generated open subgroup $B$ of $A$ over which $M$ is polycontractable. The resulting decomposition is $A$-invariant. Hence we can suppose (for the existence) that $M$ is monotypic over $B$. Then Lemma \ref{boundchainc} ensures that the number of summands in a direct decomposition of $M$ is bounded: choose a compactly generated $B'\supseteq B$ such that the monotypic disjoint decomposition has the largest possible number of summands. Then for every $B''\supseteq B'$, the monotypic disjoint decomposition is the same over $B''$ and $B'$. To show that this is the monotypic decomposition over $A$, we show the following:

Suppose that $S,S'$ are non-isomorphic simple contractable $A$-modules. Then there is a compactly generated open subgroup $B$ of $A$ such that $S,S'$ are not isomorphic as $B$-modules. If there is no common contracting element, the result is clear. So let $\alpha\in A$ be a common contracting element for $S$ and $S'$, and fix a compactly generated subgroup $B_0$ containing $\alpha$.
Then, for $B$ containing $B_0$, we have the submodule $V_B=\mathrm{Hom}_B(S,S')$ of $S\times S$. If $B\subseteq B'$ then $V_B\supseteq V_{B'}$. By Lemma \ref{boundchainc}, there exists a compactly generated subgroup $B\subseteq A$ containing $B_0$ such that for every compactly generated $B'\subseteq A$ containing $B$, we gave $V_B=V_{B'}$. It follows that $V_A=V_B$. Since $V_A=\{0\}$, we deduce $V_B=0$. 
\end{proof}

Another coarser decomposition is the characteristic one.

\begin{prop}[Characteristic decomposition]\label{chardec1}
Every polycontractable LC $A$-module uniquely decomposes in a unique way as a topological direct sum of modules
$$M=M^\circ\oplus \left(\bigoplus_p M_p\right) \oplus \left(\bigoplus_p N_p\right),$$
where $p$ ranges over primes, only finitely of the summands are nonzero, $p^{n_p}M_p=0$ for some $n_p\ge 1$, $N_p$ is canonically a finite-dimensional vector space over $\Q_p$, and $M^\circ$ is a Euclidean group.
\end{prop}
\begin{proof}
Uniqueness is clear, so let us focus on existence. Proposition \ref{monotypic} reduces to the case of a monotypic contractable module $M$, with simple quotient $S\simeq\K$, $\K$ a nondiscrete locally compact field endowed with a homomorphism $A\to\K^*$. Then $M$ has a finite composition series with each subquotient isomorphic to $S$. If $\K$ has finite characteristic, say characteristic $p$, then $p^nM=0$ for some $n$, and hence the conclusion holds. If $\K$ has characteristic zero, it is a finite extension of the closure of $\Q$, which is isomorphic to either $\R$ or $\Q_p$ for some $p$. In all cases, the conclusion holds.
\end{proof}

\begin{defn}
We say that a locally compact abelian group $V$ is {\bf characteristically pure} if one of the following holds:
\begin{itemize}
\item $p^nV=\{0\}$ for some prime $p$ and $n\ge 1$. We then say that the characteristic is $(p,p)$;
\item $V$ is a finite-dimensional vector space over $\Q_p$. We then say that the characteristic is $(0,p)$;
\item $V$ is a finite-dimensional real vector space. We then say that the characteristic is $(0,0)$.
\end{itemize}
\end{defn}

The ``characteristic" is of course unique, unless $V=\{0\}$. From Proposition \ref{chardec1} it follows that every indecomposable polycontractable module is characteristically pure.

For $A=\langle\alpha\rangle$ infinite cyclic, a contractable module is contracted by either $\alpha$ or $\alpha^{-1}$. In finite characteristic we have a simple classification result.

\begin{thm}\label{indecchp}
Suppose that $A=\langle\alpha\rangle$ is infinite cyclic. Let $M$ be an indecomposable module of finite characteristic on which $\alpha$ acts as a contraction. Then there exists a prime $p$ and $n\ge 1$ such that $M$ is isomorphic to $(\Z/p^n\Z)\lp t\rp$, with $\alpha$ acting by multiplication by $t$.
\end{thm}
\begin{proof}
There exists a prime $p$ and $n\ge 1$ (chosen minimal) such that $p^nA=\{0\}$. Hence we see that $M$ is naturally a module over the ring $(\Z/p^n\Z)\lp t\rp$, with $\alpha$ acting by multiplication by $p$. (More precisely, if we consider the unique ring homomorphism $\Z[\alpha^{\pm 1}]\to (\Z/p^n\Z)\lp t\rp$ mapping $\alpha\mapsto t$, we see that the action $\Z[\alpha^{\pm 1}]\times M\to M$ uniquely factors through a continuous action $(\Z/p^n\Z)\lp t\rp\times M\to M$.)

Now we observe that $R_n=(\Z/p^n\Z)\lp t\rp$ is a principal ideal local ring, with its ideals being the ideals $(p^k)$ for $0\le k\le n$. Indeed, first we see that every element of the form $1+x$ with $x\in t(\Z/p^n\Z)[\![ t]\!]$ is invertible. Now every element $y$ of $R_n\smallsetminus pR_n$ can be written as $y=t^k\lambda(s+1+x)$ with $k\in\Z$, $s$ nilpotent, $\lambda\in(\Z/p^n\Z)^\times$ and $x\in t(\Z/p^n\Z)[\![ t]\!]$. Indeed, $k$ is the smallest degree at which the coefficient of $y$ is not a multiple of $p$, $\lambda$ is the given coefficient. Then $s+(1+x)$, as sum of a nilpotent and an invertible, is invertible.

It follows that every element of $R_n$ can be written as $p^kz$ with $k\ge 0$ and $z$ invertible, showing the desired result.

From the classification of modules over principal ideal rings\footnote{We can, for instance, use the fact that every local principal ideal ring is a quotient of a local PID. Here this can be done explicitly, namely defining $R_\infty$ as the ring of formal series $\sum_{n\in\Z}a_nt^n$ with $a_n\in\Z_p$ and for which the $p$-valuation of $a_n$ tends to $\infty$ when $n$ tends to $-\infty$. Alternatively, the classification of finitely generated modules can be directly established without appealing to the case of domains.}, we deduce that every finitely generated indecomposable $R_n$-module has the form $R_n/(p^k)$ for some $k$. Since $n$ is minimal for the property $p^nM=\{0\}$, we deduce that $M$ is isomorphic to $R_n$ (viewed as the free $R_n$-module of rank one).
\end{proof}

\begin{rem}
By Theorem \ref{indecchp}, there are, up to isomorphism, countably many torsion contractable LC modules over an infinite cyclic group. This is in contrast with characteristic zero: indeed, the action on $\mathbf{Q}_p$ with $\alpha$ acting by an element $x\in\mathbf{Q}_p^*$ of modulus $\neq 1$ defines a simple contractable module $M_x$, and the modules $M_x$ are pairwise non-isomorphic, and there are continuum many.

For $A=\Z^2=\langle \alpha,\beta\rangle$ we also get continuum many simple contractable modules in characteristic $p$, by letting $A$ act on $(\Z/p\Z)\lp t\rp$, $\alpha$ acting by multiplication by $t$ and $\beta$ by a nonzero element $x$. The resulting modules $M_x$ are pairwise non-isomorphic.
\end{rem}


Retrospectively, most of the theory can be reduced to $A=\Z$ by the following proposition.

\begin{prop}
Let $M$ be a polycontractable LC $A$-module. Then there exists $\alpha\in A$ such that $M$ is polycontractable as $\Z[\alpha^{\pm 1}]$-module. More precisely, there exists $n$ and nonzero continuous homomorphisms $f_1,\dots,f_n:A\to\R$ such that every $\alpha\notin\bigcup_{i=1}^n\Ker(f_i)$ works.
\end{prop}
\begin{proof}
Write $M=\bigoplus M_i$ with $M_i$ monotypic. Let $S_i$ be a simple quotient of $M_i$: it can be identified with a local field, with norm $\|\cdot\|_i$, on which $A$ acts through a homomorphism $u_i$ into the multiplicative group, such that the homomorphism $f_i=\log\circ \|\cdot\|_i\circ u_i$ is nonzero. Then any element of $A$ not in the kernel of $f_i$ acts on $S_i$, and hence on $M_i$, as a contraction or the inverse of a contraction. Whence the conclusion.
\end{proof}

Let us now provide a cocompactness criterion.

\begin{thm}\label{coccri}
Suppose that $A$ is compactly generated. Let $M$ be a sheer LC module. Write $M=D\oplus \bigoplus_{i=0}^nM_i$ with $D$ distal Euclidean, $M_0$ amorphic, each $M_i$ polycontractable with the $M_i$ pairwise disjoint. Let $H$ be a closed submodule of $M$.

(1) Suppose that the sheer envelope $H^\boxplus$ of $H$ (see Theorem \ref{corenv}) equals $M$. Then the projection of $H$ on each $M_i$ is dense, and the projection of $H$ on $D$ spans $D$.

(2) Suppose that the projection of $H$ on each $M_i$ is dense, and the projection of $H$ on $D$ spans $D$. Then $H$ is cocompact in $M$.
\end{thm}
\begin{proof}
(1) Let $H_i$ be the closure of the projection of $H$ in $M_i$, and let $V$ be the span of the projection of $H$ in $D$. Then $V\oplus \bigoplus_{i=0}^nH_i$ is sheer and contains $H$. The assumption implies that it equals $M$, whence the conclusion.

(2) Since $H^\boxplus$ is sheer, the assumptions ensure that $H^\boxplus$ decomposes along the same decomposition: $H^\boxplus=(H^\boxplus\cap D)\oplus \bigoplus_{i=0}^n(H^\boxplus\cap M_i)$.
Since the projections of $H$ are dense and $H\subseteq H^\boxplus$, we have $H^\boxplus\cap M_i=M_i$ for all $i$. Also, since $H^\boxplus\cap D$ spans $D$ and is sheer, we have $H^\boxplus \cap D=D$. Hence $H^\boxplus=M$. Now Theorem \ref{corenv} says that $H^\boxplus/H$ is compact.
\end{proof}


\section{Metabelian groups}

\subsection{Compact generation and ascending sequences}\label{metab-cg}

We begin by the following result, which is very classical in the discrete case, and is a particular case of \cite[Satz 2.1.b]{Ab72}.

\begin{lem}\label{cgcg}
Let $G$ be a compactly generated metabelian LC-group with a closed normal abelian subgroup $M$ such that $A=G/M$ is abelian. Then $M$ is a compactly generated $\Z A$-module.
\end{lem}
\begin{proof}Since $G$ is compactly generated and $G/M$ is compactly presentable, the kernel $M$ is compactly generated as normal subgroup (see \cite[Prop.\ 8.A.10(2)]{CH}).
\end{proof}



The following lemma will be superseded by Corollary \ref{ncgme}, but we need it at this point as it will be used to prove the latter.

\begin{lem}\label{centcg}
Let $G$ be a compactly generated metabelian locally compact group.  Let $N$ be a closed central subgroup of $G$. Then $N$ is a compactly generated locally compact group.
\end{lem}
\begin{proof}
Let $M$ be an abelian closed normal subgroup such that $G/M$ is abelian. The assumptions imply that $M'=\overline{MN}$ is an abelian normal subgroup. Hence by Lemma \ref{cgcg}, $M'$ is compactly generated as $\Z[G/M']$-module. By Corollary \ref{scgcg}, it follows that its closed submodule $N$ is compactly generated as $\Z[G/M']$-module. Since it is central, it is thus abstractly compactly generated.
\end{proof}

\begin{thm}\label{menoeth}
Let $G$ be a compactly generated locally compact metabelian group, $(N_n)$ an ascending sequence (or more generally, filtering net) of closed normal subgroups of $G$. Define $N=\overline{\bigcup N_n}$. Then there exists $n_0$ such that $N_n$ is cocompact in $N$ for all $n\ge n_0$.
\end{thm}
\begin{proof}
First, when $G$ is abelian, the result can be viewed as a particular case of Theorem \ref{ascending} (for the action of the trivial group).

Let $M$ be a closed normal abelian subgroup of $G$ such that $A=G/M$ is abelian. 

We begin with the particular case when $N_n\cap M$ does not depend on $n$, say is equal to $L$. Then modding out by $L$ if necessary, we can suppose that $N_n\cap M=\{1\}$ for all $n$. Since $[G,N_n]\subseteq N_n\cap M$ for all $n$, we deduce that $N_n$ is central for all $n$. Hence $N$ is central in $G$ as well. By Lemma \ref{centcg}, it follows that $N$ is compactly generated as well. By the abelian case, there exists $n_0$ such that $N_n$ is cocompact in $N$ for all $n\ge n_0$.

In general $M_n=N_n\cap M$ and $V=\overline{\bigcup_n M_n}$. Then $(M_n)$ is an ascending sequence of closed $A$-submodules of the compactly generated $A$-module $M$, so Theorem \ref{ascending} ensures that for large $n$, $M_n$ is cocompact in $V$. Write $N=\overline{\bigcup_n VN_n}$. Then $VN_n$ is closed and intersects $M$ in $V$, and for $n$ large enough, $N_n$ is coompact in $VN_n$. Also, the previous paragraph applies: for $n$ large enough, $VN_n$ is cocompact in $N$. Hence, for $n$ large enough, $N_n$ is cocompact in $N$.
\end{proof}


Similarly to Corollary \ref{scgcg}, we deduce:

\begin{cor}\label{ncgme}
If $G$ is a compactly generated locally compact metabelian group, then every normal subgroup $N$ is compactly generated as a normal subgroup.
\end{cor}
\begin{proof}
Let $N_0$ be an open subgroup of $N$, compactly generated as a normal subgroup. Since $G$ and hence $N$ is $\sigma$-compact, $N/N_0$ is countable, hence choose an ascending sequence $(N_n)$ of open subgroups, each open and compactly generated as normal subgroup, whose union is $N$. By Theorem \ref{menoeth}, there exists $n$ such that $N_n$ is cocompact in $N$. Thus the compactly generated submodule $N_n$ has finite index in $N$, and hence $N$ is compactly generated as well. 
\end{proof}

\begin{thm}\label{cgmmax}
Let $G$ be a compactly generated metabelian locally compact group. Then
\begin{enumerate}
\item\label{tm1} $G$ has a maximal compact normal subgroup $\Ell(G)$, and $\Ell(G/\Ell(G))=1$.
\item\label{tm2} If $W(G)=1$ and $G$ is totally disconnected, then $[G,G]$ is closed.
\end{enumerate}
\end{thm}
\begin{proof}
Write  $M=\overline{[G,G]}$ and $A=G/M$.

(\ref{tm1}) That $\Ell(G)$ compact implies $\Ell(G/\Ell(G))$ trivial is straightforward (for an arbitrary LC group). Hence we need to prove that $\Ell(G)$ is compact. By Corollary \ref{maxcp}, there is a maximal compact $\Z A$-submodule $\Ell_G(M)$ in $M$; we can mod out and assume it is trivial, so that $M\cap\Ell(G)=\{0\}$. 

Since $[G,\Ell(G)]\subseteq M\cap\Ell(G)$, we deduce that $\Ell(G)$ is central. By Lemma \ref{centcg}, the central subgroup $N=\overline{\Ell(G)}$ is a compactly generated LCA group. By density of $\Ell(N)$ in $N$, we have $\Hom(N,\R)=\{0\}$. Being a compactly generated LCA group with $\Hom(N,\R)=\{0\}$, we see that $N$ is compact, so $N=\Ell(G)$ is compact.


(\ref{tm2}) Since $\Ell(G)$ is trivial, $\Ell(M)$ is trivial as well. In particular, $M$ is lower-sheer. Hence the sheer module $\Omega^\sharp(M)$ can be written as $P\times E$ with $P$ polycontractable, $E=\Ell(M)$ amorphic (Theorem \ref{sheerdec1}). Since $\Ell(M)$ is trivial, we have $E=\{0\}$, so $M=P$ is polycontractable. It follows from Lemma \ref{difpro1} that $M\subseteq [G,G]$. Hence $[G,G]$ is open in $M=\overline{[G,G]}$ and hence $[G,G]$ is closed.
\end{proof}

\subsection{Compact presentability}

Let $Q$ be a finitely generated abelian group and $v:Q\to\R$ an epimorphism. Define the closed submonoid $Q_v=\{q\in Q:v(q)\ge 0\}$. 


\begin{lem}\label{genqgen}
Let $Q$ be an a group, and $M$ a LC group with an action of $Q$, and fix $\chi\in Q^*$. Suppose that $M$ is compactly generated as $Q_\chi$-group, and generated, as $Q$-group by a compact subset $S$ with nonempty interior. Then $M$ is generated by $S$ as $Q_\chi$-group.
\end{lem}
\begin{proof}
If $\chi\neq 0$ there is nothing to prove; assume otherwise. Let $N$ be the $Q_\chi$-subgroup generated by $S$; it is open. Fix $u\in Q$ with $\chi(u)<0$. Then $M=\bigcup_{n\ge 0}u^{-n}(N)$. This ascending union is either constant or strictly increasing, and the latter is excluded since $M$ is compactly generated as $Q_\chi$-group. Hence $N=M$.
\end{proof}

\begin{proof}This is trivial if $\chi=0$; assume otherwise.
Pick $q\in Q\smallsetminus Q_v$ (so $v(q)<0$). Let $N$ be the $Q_v$-submodule generated by $S$. Then we have $M=\bigcup_{n\ge 0} q^nN$, where each $q^nN$ is an open $Q_v$-submodule of $M$, and $q^nN\subseteq q^{n+1}N$ for each $n$. Since $M$ is compactly generated as $Q_v$-module, we have $q^nN=q^{n+1}N$ for some $n$ and hence $N=qN$; thus $N$ is a $Q$-submodule and $N=M$. 
\end{proof}

\begin{rem}
The particular case when $M$ is abelian and $M$ is discrete was obtained by Bieri and Strebel \cite[Proposition 2.1]{BS}. The above proof differs from the original proof: indeed, the latter relies on a criterion (written in \cite[Proposition 2.1]{BS}) for $M$ to be compactly generated over $Q_v$, which is stated as the existence of a certain nonzero element (with suitable properties) in the annihilator in $\Z[Q]$ of $M$. This criterion does not carry over to the non-discrete case, since $M$ can be compactly generated over $Q_v$ and have a trivial annihilator: for instance consider $M=\Q_p$ and $Q=\langle t\rangle$ infinite cyclic, where $t$ acts on $\Q_p$ by multiplication by some transcendental element of positive valuation, and $v$ defined by $v(t)=-1$. However, we will use this criterion in the proof of Lemma \ref{opega}.
\end{rem}

\begin{rem}
If we assume that $Q$ is countable, or more generally, $Q$ is a $\sigma$-compact locally compact group and the $Q$-action is continuous, we can, by a simple application of Baire's theorem, remove the assumption that $S$ has nonempty interior.
\end{rem}

We now need two lemmas of abstract group theory. If $G$ is a group and $u\in G$, define an $u^+$-subgroup to be a subgroup stable under the left conjugation endomorphism $x\mapsto uxu^{-1}$ and similarly define an $u^-$-subgroup; define an $u$-subgroup as a subgroup stable under conjugation by $u$ and $u^{-1}$.


%


\begin{lem}\label{condn}
Consider a group $G$ with a generating subset $S\cup\{u\}$ with $S$ symmetric. Let $W_0\subseteq G$ be a subgroup and define $W_n=u^nWu^{-n}$ for $n\in\Z$.
Assume that the following conditions hold
\begin{enumerate}
\item\label{rel11} for every $n\ge 0$, $S$ normalizes $W_n$;
\item\label{rel12} $uWu^{-1}\subseteq W_0$;
\item\label{rel2} $uSu^{-1}$ is contained in the $u^-$-subgroup generated by $S$.
\end{enumerate}
Then the normal subgroup $N$ generated by $W_0$ is the ascending union $W_\infty=\bigcup_{n\ge 0} u^{-n}W_0u^n$.


\end{lem}
\begin{proof}
Observe that because of Condition (\ref{rel12}), the union $W_\infty=\bigcup_{n\ge 0} W_{-n}$ is indeed ascending, and hence is a subgroup since $W_n$ is a subgroup.

Next, we see by a straightforward induction (only based on (\ref{rel2})) that for every $y\in S$ and $n\ge 0$, the element $u^nyu^{-n}$ belongs to the $u^-$-subgroup generated by $S$.


Observe that for every $n\ge 0$, $u^{-n}Su^n$ normalizes $W_0$: indeed, for $s\in S$ and $w\in W_0$, we have
\[^{u^{-n}su^n}W_0={}^{u^{-n}s}W_n={}^{u^{-n}}W_n=W_0.\]
Since by the previous fact, the normal subgroup $\langle\!\langle S\rangle\!\rangle$ generated by $S$ is generated by $\bigcup_{n\ge 0}u^{-n}Su^n$, we deduce that $\langle\!\langle S\rangle\!\rangle$ normalizes $W_0$. Since $\langle\!\langle S\rangle\!\rangle$ is normal, it therefore normalizes every conjugate of $W$, hence normalizes each $W_n$, and hence normalizes $W_\infty$. Since $G=\langle u\rangle \langle\!\langle S\rangle\!\rangle$, it therefore normalizes $W_\infty$.
\end{proof}




\begin{lem}\label{condn2}
Keep the assumptions of the preceding lemma. Assume that $G$ carries a topology of locally compact group such that $S,W$ are compact, $N$ is closed, and $S$ has non-empty interior. Then if $G/N$ is compactly presentable, so is $G$. 
\end{lem}
\begin{proof}
We can suppose that $W\subseteq S$, and fix $L=S\cup\{u\}$ as generating subset. Consider relators of bounded length for $G/N$, and lift them to relators of the form 
$r\equiv u^{-n}wu^{n}$ with $n$ bounded.

By a compactness argument, there exist $n_0$ and $k_0$ such that $uSu^{-1}$ is contained in the set of products of at most $n_0$ elements of the form $\bigcup_{0\le k\le k_0}u^{-k}Su^k$. Prescribe them as relators as well.

Include the relators $xwx^{-1}=w'$ for $(x,w,w')\in L\times W^2$ as well.

Let $G'$ be the group defined by this presentation. By construction, the canonical projection $p:G'\to G$ is injective on each of $S$, and $W$ (which are viewed both as subsets of $G'$ and $G$, as well as the element $u$). By Lemma \ref{condn}, the normal subgroup $N$ of $G'$ generated by $W$ is the ascending union $\bigcup_{n\ge 0}u^{-n}Wu^n$. It follows that $N$ is mapped injectively onto its image in $G$, also called $N$. Since the relators include lifts of relators of $G/N$, the mapped $G'/N\to G/N$ induced by projection $p$ is an isomorphism. Therefore $p$ is a group isomorphism as well. Since relators have bounded size, this shows that $G$ is compactly presentable.
\end{proof}

\subsection{The Bieri-Strebel geometric invariant}

Let us introduce some basic definitions in the locally compact setting. If $G$ is a locally compact group, denote $G^*=\Hom(G,\R)$ the set of continuous homomorphisms $G\to\R$, endowed with the compact-open topology. When $G$ is compactly generated, this is a finite-dimensional real vector space, with its usual topology. If $\chi\in G^*$, define $G_\chi=\{g\in G:\;\chi(G)\ge 0\}$. 


\begin{defn}[Bieri-Neumann-Strebel invariant in the locally compact setting]
Let $H$ be a locally compact group with a continuous action of $G$ by automorphisms. Define
\begin{gather*} \Gamma_G(H)=\{0\}\cup\{\chi\in G^*:\;H\textnormal{ is not compactly generated}\\
\textnormal{over any compactly generated submonoid of }G_\chi\},
\end{gather*}
In particular, define
\[\Gamma(G)=\Gamma_G(\overline{[G,G]}).\]
\end{defn}

This is introduced in \cite{BNS} when $G,H$ are discrete,, in which case this is a closed subset of $G^*$, stable under positive homotheties \cite[Theorem A]{BNS}.

(The invariant defined in \cite{BS,BNS} is rather the quotient of $G^*\smallsetminus\Gamma(G)$ by the group of positive homotheties; while these are obviously equivalent data, it is much more practical not to use the complement, for instance so as to state the next lemma.)

We will check (Lemma \ref{eqmet}) that for $G$ metabelian, this matches with the original definition in \cite{BS}, namely, for a closed abelian normal subgroup $M$ with $Q=G/M$ abelian, we have a natural identification $\Gamma(G)=\Gamma_Q(M)$. (Such an identification would not be possible by working with the open complement in the definition!)


\begin{lem}\label{cgsousq}
Let $W,Q'\subseteq Q$ be closed subgroups, with $W$ compact, such that $Q=W+Q'$. Let $\chi:Q\to\R$ be a continuous homomomorphism.
Then every LC $Q$-module $M$ that is compactly generated as $Q_\chi$-module is compactly generated as $Q'_\chi$-module.
\end{lem}
\begin{proof}
Since $\chi$ vanishes on $W$, we have $Q_\chi=Q'_\chi+W$. Let $K$ be a compact subset generating $M$ as $Q_\chi$-module. Then $WK$ is compact and generates $M$ as $Q'$-module. 
\end{proof}

\begin{lem}\label{congesu}
Let $Q$ be a $\sigma$-compact LC abelian group and $M$ a locally compact group with a continuous action of $Q$. Consider $\chi\in Q^*$. Suppose that $M$ is compactly generated as a $Q_\chi$-group. Then $M$ is compactly generated over some compactly generated subsemigroup of $Q_\chi$.
\end{lem}
\begin{proof}
Let $S$ be a compact generating subset of $M$ as a $Q$-group, with non-empty interior. By Lemma \ref{genqgen}, $S$ generates $M$ as a $Q_\chi$-group. Fix $u\in Q$ with $\chi(u)<0$. 
Write $Q_\chi$ as an ascending union of compactly generated subsemigroups $P_k$, such that $P_0\cup\{u\}$ generates $Q$ as a semigroup. Thus if $N_k$ is the $P_k$-subgroup generated by $S$, we have $M=\bigcup_{n\ge 0} u^{-n}N_k$ for each $k$. Besides, we have $\bigcup_{k\ge 0} N_k=M$. In particular, by compactness, $u^{-1}S$ is contained in $M_k$ for some $k$. By an immediate induction (using that $Q$ is abelian), it follows that $u^{-n}S$ is contained in $N_k$ for all $n$. We deduce that $N_k=M$.
\end{proof}

The following two lemmas will appear as particular cases of Theorem \ref{chisous}, but we need them for the moment.

\begin{lem}\label{opega}
Let $Q$ be a compactly generated abelian group and $M$ a compactly generated $\Z[Q]$-module. Fix $\chi\in Q^*$ and suppose that $M$ is compactly generated as a $Q_\chi$-module. Let $N$ be an open $\Z[Q]$-submodule of $M$. Then $N$ is compactly generated as a $Q_\chi$-module.
\end{lem}
\begin{proof}
Let $W$ be the maximal compact subgroup in $Q$. Let $Q'$ be a finitely generated subgroup of $Q$ whose projection in $Q/W$ is dense.

Fix $u\in Q'$ with $\chi(u)<0$. Following \cite{BS}, for $m\in\Z[Q']$, define $\chi(m)$ as the infimum of $\inf\{\chi(q): q\in Q', m(q)\neq 0\}$. Here $m$ is viewed as a finitely supported function $Q'\to\Z$. (We agree that $\chi(0)=+\infty$.)

 By Lemma \ref{cgsousq}, $M$ is compactly generated as $\overline{Q'}_\chi$-module, and hence as $Q'_\chi$-module.
 Hence $M/N$ is finitely generated as a $Q'_\chi$-module. Therefore, by \cite[Proposition 2.1]{BS}, there exists $m\in\Z[Q']$ such that $m$ acts on $M/N$ as the identity and $\chi(m)>0$. Replacing $m$ by a large enough power, we can suppose $\chi(m)>-\chi(u)$.
 

Let $S_0$ be a compact neighborhood of $0$ in $N$, generating $N$ as $\Z[Q]$-submodule. Let $T$ be a finite subset of $M$ such that $S\cup T$ generates $M$ as a $Q_\chi$-module. Write $S=S_0\cup (1-m)T$, and observe that $S\subseteq N$. Let $N'$ be the $Q_\chi$-submodule of $N$ generated by $S$. We claim that $N'=N$.

If $s_0\in S$, using that $S\cup T$ generates $M$ as a $Q'_\chi$-module, write a finitely supported decomposition $us_0=\sum_{s\in S}a_ss+\sum_{t\in T}b_tt$, with $a_s,b_t\in\Z[Q'_\chi]$. Write $(1-m)t=\sum_{s\in S}c_ss$, so that $(1-m)\sum_{t\in T}b_tt=
\sum_{s\in S}\sum_{t\in T} b_t c_s s$. Write $d_s=(1-m)a_s+c_s\sum_{t\in T}b_t$. Then $(1-m)ux=\sum_{s\in S}d_ss$. Thus $us_0=mus_0+\sum_{s\in S}d_ss$. This shows that $uS\subseteq N'=\Z[Q'_\chi]S$. By induction, $u^nS\subseteq N'$ for all $n\ge 0$. Since every element of $Q'$ can be written as $u^nq$ with $q\in Q'_\chi$ and $n\ge 0$, we deduce that $N'$ is a $\Z[Q']$-submodule, hence equals $N$.
\end{proof}

\begin{lem}\label{closedcotrivial}
Let $Q$ be a compactly generated abelian group and $M$ a compactly generated $\Z[Q]$-module. Fix $\chi\in Q^*$ and suppose that $M$ is compactly generated as a $\Z[Q_\chi]$-module. Let $N$ be a closed $\Z[Q]$-submodule of $M$ such that $Q$ acts trivially on $M/N$. Then $N$ is compactly generated as a $\Z[Q_\chi]$-module.
\end{lem}
\begin{proof}
By Corollary \ref{scgcg}, $N$ is compactly generated as $Q$-module. By Corollary \ref{maxcp}, $\Ell(N)$ and $\Ell(M)$ are compact. Modding out if necessary, we can suppose that $\Ell(N)=\{0\}$. Then $N\cap\Ell(M)=\{0\}$, so, modding out again, we can also suppose that $\Ell(M)=\{0\}$. 
So both $N,M$ are lower-sheer. Using Theorem \ref{sheerdec1}, write $\Omega^\sharp(M)=P\times D$, $\Omega^\sharp(M)=P\times D$ with $P,P'$ polycontractable, $D,D'$ distal Euclidean as $Q$-modules. Then $P'\subseteq P$ and $D'\subseteq D$. Then $P/P'$ embeds into $M/N$ which is a module with trivial action. Since $P$ is generated by contractable submodules, it has no nonzero homomorphism into $M/N$, which means that $P\subseteq N$ and thus $P=P'$. By Lemma \ref{opega}, the open submodule $\Omega^\sharp(M)$ is compactly generated as $Q_\chi$-module, and hence so is its quotient $P$. Hence it is enough to show that $N/P$ is compactly generated as $Q_\chi$-module. Indeed, $\Omega^\sharp(N)/P$ is closed in $D$ and hence compactly generated as abelian group. Finally, the discrete module $N/\Omega^\sharp(N)$ is open in $M/\Omega^\sharp(M)$ and therefore is also, by (the discrete case of) Lemma \ref{opega}, compactly generated.
\end{proof}

\begin{lem}\label{eqmet}
Let $G$ be a compactly generated metabelian group in an exact sequence 
\[1\to M\to G\to Q\to 1\]
with $M,Q$ abelian and $M$ closed. Then $M$ is compactly generated as a $Q$-module, and 
\[\Gamma(G)=\Gamma_Q(M),\]
where $Q^*$ is naturally included in $G^*$.
\end{lem}
\begin{proof}
That $M$ is a compactly generated $Q$-module is already asserted in Lemma \ref{cgcg}.

Let us first check $\Gamma_Q(M)\subseteq\Gamma(G)$. Let $\chi$ belong to $\Gamma_Q(M)$. If by contradiction $\chi\notin\Gamma(G)$, then $\chi\neq 0$ and $\overline{[G,G]}$ is compactly generated over some compactly generated subgroup of $G_\chi$, and hence over $Q_\chi$.

Since $Q$ is abelian, we have $\overline{[G,G]}\subseteq M$. In particular, $M$ being abelian, it centralizes $\overline{[G,G]}$. Hence $\overline{[G,G]}$ is compactly generated over $Q_\chi$. Since $M/\overline{[G,G]}$ is isomorphic to a closed subgroup of $G/\overline{[G,G]}$, it is compactly generated; hence $M$ is compactly generated over $Q_\chi$. This contradicts the assumption, so $\chi\in\Gamma(G)$.

Conversely, let us show $\Gamma(G)\subseteq\Gamma_Q(M)$. We first check $\Gamma(G)\subseteq Q^*$. Indeed, let $\chi$ belong to $G^*\smallsetminus Q^*$. This means that $\chi$ is not zero on $M$. In particular, $MG_\chi=G$. Thus a closed subgroup of $M$ is $G_\chi$-invariant if and only if it is $G$-invariant. Since $M$ is compactly generated over $G$, this implies that it is also compactly generated over $G_\chi$, and hence $\chi\notin\Gamma(G)$. We have proved $\Gamma(G)\subseteq Q^*$.
 
Now let $\chi$ belong to $Q^*\smallsetminus\Gamma_Q(M)$. Hence $\chi\neq 0$ and $M$ is compactly generated as a $Q_\chi$-module. Since $\overline{[G,G]}$ is a closed submodule of $M$ and the action on the quotient $M/\overline{[G,G]}$ is trivial, Lemma \ref{closedcotrivial} ensures that $\overline{[G,G]}$ is compactly generated as $Q_\chi$-module. Thus $\chi\notin\Gamma(G)$.
\end{proof}

Theorem \ref{chisous} will remove the assumption below that $N$ is open.

\begin{lem}\label{extgam}
Let $Q$ be a compactly generated abelian group and $M$ a compactly generated $\Z[Q]$-module. Assume that $N$ is a closed submodule. Then \[\Gamma_Q(M/N)\subseteq \Gamma_Q(M)\subseteq\Gamma_Q(N)\cup\Gamma_Q(M/N).\] If $N$ is open in $M$, then $\Gamma_Q(M)=\Gamma_Q(N)\cup\Gamma_Q(M/N)$.
\end{lem}
\begin{proof}
Clearly $\Gamma_Q(M/N)\subset\Gamma_Q(M)$. The inclusion $\Gamma_Q(N)\subset\Gamma_Q(M)$ follows from Lemma \ref{opega} (we only use there that $N$ is open). The reverse inclusion follows from the fact that an extension of compactly generated modules (over $Q_\chi$ for given $\chi$) is compactly generated.
\end{proof}

\begin{defn}If $Q$ is a compactly generated abelian group, we have an inclusion $\Hom(Q,\Q)\subset\Hom(Q,\R)=Q^*$ (with $\Q$ endowed with the discrete topology), $\Hom$ denoting continuous homomorphisms; we call its elements the {\bf rational characters}; if $\chi$ is a rational character, the closed half-line $\R_+\chi$ is called a {\bf rational half-line} in $Q^*$. 
\end{defn}

In general, the closure in $Q^*$ of the set of rational characters is the subspace $(Q/Q^\circ)^*$ of $Q^*$ consisting of homomorphisms vanishing on the component $Q^\circ$. Note that $\chi\in Q^*$ is a scalar multiple of a rational character if and only if has a discrete image in $\R$.

\begin{lem}\label{gamono}
Let $Q$ be a compactly generated abelian group. Let $M$ be a monotypic contractable $Q$-module, with simple quotient $\K$, a nondiscrete locally compact field endowed with a continuous homomorphism $\phi:Q\to\K^*$. Write $\chi_0=-\log |\phi|$. Then
\begin{itemize}\item if $M$ is connected, $\Gamma_Q(M)=\{0\}$;
\item if $M$ is totally disconnected then $\Gamma_Q(M)=\R_+\chi_0$, which is a rational half-line. Moreover, $M$ admits an open compact $Q_{\chi_0}$-submodule.
\end{itemize}
\end{lem}
\begin{proof}
(The description of $\K$ is provided by Corollary \ref{sico}.)

Note that being monotypic contractable, $M$ is either connected or totally disconnected. If $M$ is connected it is compactly generated as locally compact abelian group and the conclusion is clear, so assume that $M$ is totally disconnected.


Each ball in $\K$ is preserved by $Q_{\chi_0}$, since the latter acts on $K$ by elements of norm $\le 1$. Hence $\K$ is not a compactly generated $Q_{\chi_0}$-module. Since $\K$ is a quotient of $M$ as $Q_{\chi_0}$-module, the latter is not compactly generated. So $\chi_0\in\Gamma_Q(M)$.

Conversely, let $\chi$ be an element in $Q^*\smallsetminus \R_+\chi_0$. Then there exists $\alpha\in Q$ with $\chi(\alpha)>0$ (i.e., $\alpha\in Q_\chi$) and $\chi_0(\alpha)<0$ (i.e., $|\phi(\alpha)|>1$). Then for every compact neighborhood $V$ of $0$ in $\K$ we have $\K=\bigcup_{n\ge 0}\phi(\alpha^n)V$. In particular, $\K$ is compactly generated over $Q_\chi$. Since $M$ is an iterated extension of copies of the $Q$-module $\K$, we deduce that $M$ is compactly generated over $Q_\chi$ as well, i.e., $\chi\notin\Gamma_Q(M)$.

Since $\chi_0$ has a discrete image, $\R_+\chi_0$ is a rational half-line.

Let us check the additional assertion of existence of an open $Q_{\chi_0}$-submodule. Choose $u\in Q$ such that $\chi_0(u)$ is the positive generator of $\chi_0(Q)$. Then $Q_{\chi_0}$ is generated by $\Ker(\chi_0)\cup\{u\}$ as a semigroup. As a $\Ker(\chi_0)$-module, $\K$ and hence $M$ is amorphic. So there is a compact open $\Ker(\chi_0)$-submodule $L\subseteq M$. Since $u$ acts as a contraction on $\K$ and hence on $M$, for large enough $n$, the submodule $u^nL$ is contained in $L$. Hence $\sum_{n\ge 0}u^nL$ is a compact open $Q_{\chi_0}$-submodule of $M$.
\end{proof}

\begin{prop}\label{upsheg}
Let $Q$ be a compactly generated abelian group. Let $M$ be an upper-sheer compactly generated locally compact $Q$-module. Then $\Gamma_Q(M)$ is a finite union of rational half-lines (emanating from 0). More precisely, for $\chi\in Q^*\smallsetminus\{0\}$, $\chi\in\Gamma_Q(M)$ if and only if there is a totally disconnected simple contractable closed submodule $\K$ of $M$ with a normed field structure, such that the action is given by some homomorphism $\phi:Q\to\K^*$, and such that $\chi$ is positively proportional to $-\log|\phi|$.
\end{prop}
\begin{proof}
If $M$ is polycontractable, it is a finite direct product of contractable monotypic modules and the result follows from Lemma \ref{gamono}.

In general, modding out by $\Ell(M)$ (which is compact, by Corollary \ref{maxcp}), we can suppose that $\Ell(M)=\{0\}$. So $M$ is sheer and by Theorem \ref{sheerdec1} it can be written as $P\times D$, with $P$ polycontractable and $D$ distal Euclidean. Since $\Gamma_Q(M)=\Gamma_Q(M/D)=\Gamma_Q(P)$, we are done.
\end{proof}

\begin{prop}\label{ratpo}
Let $Q$ be a compactly generated abelian locally compact group. Let $M$ be a compactly generated locally compact $Q$-module.
Then $\Gamma_Q(M)$ is a rational polyhedral cone, i.e., is the union of finitely many convex polyhedral cones (based at one) whose faces have rational equations.
\end{prop}
\begin{proof}
Let $N$ be an upper sheer open submodule of $M$. Then $\Gamma_Q(M/N)$ is rational polyhedral by Bieri-Groves's theorem \cite{BG}, and $\Gamma_Q(N)$ is a finite union of rational half-lines by Proposition \ref{upsheg}. Then Lemma \ref{extgam} ensures that $\Gamma(M)$ is the union of $\Gamma_Q(M/N)$ with a finite union of rational half-lines.
\end{proof}

\begin{lem}\label{gammasheer}
Let $Q$ be a compactly generated abelian group. Let $M$ be a sheer module and $N$ a closed submodule. Then $\Gamma_A(N)\subseteq\Gamma_A(M)$.
\end{lem}
\begin{proof}
The first case is when $N$ is polycontractable. This case directly follows from Proposition \ref{upsheg}.

In general, let $N^\boxplus$ be the sheer envelope, as in Theorem \ref{corenv}. By the polycontractable case, $\Gamma_A(N^\boxplus)\subseteq\Gamma_A(M)$. This thus reduces to the case when $N^\boxplus=M$, which we now assume.




Let $\chi$ be an element of $\Gamma_Q(N)$, and write $Q'=\Ker(\chi)$. Since $M$ is a sheer $Q$-module, it is a sheer $Q'$-module. Using Theorem \ref{sheerdec1}, write $N=D\oplus M_0\oplus P$ with $D$ distal Euclidean $Q'$-module, $M_0$ amorphic $Q'$-module, and $P$ polycontractable $Q'$-module. In turn, write the monotypic decomposition (Proposition \ref{monotypic}) of the $Q'$-module $P$ as $P=\bigoplus_{i=1}^nM_i$. Write $D=M_{-1}$.

Let $u$ be an element of $Q$ with $\chi(u)<0$. Then there is a proper $Q_\chi$-submodule $L$ of $N$ such that $N=\bigcup_{n\ge 0}u^{-n}L$ (namely, one can take $Q_\chi$-submodule of $N$ generated by any compact generating subset of $N$ as $Q$-module with nonmpty interior in $N$). For $i\ge 0$, let $L_i$ (resp.\ $N_i$) be the closure of the projection of $L$ (resp.\ $N$) on $M_i$. For $i=-1$, we define $L_{-1}$ (resp.\ $N_{-1}$) be the span of the projection of $L$ (resp.\ $N$) on $M_{-1}$. By Theorem \ref{coccri}(1), $N_i=M_i$ for all $i\ge -1$.

 Then $\bigcup_{n\ge 0}u^{-n}L_i=N_i=M_i$, and thus $L_i$ is an open $Q_\chi$-submodule of $M_i$. 

Suppose by contradiction that for all $i$, we have $L_i=M_i$. By Theorem \ref{coccri}(2), $L$ is cocompact in $M$. But $L$ is open of infinite index in $N$. We reach a contradiction. 

Hence, there exists $i\ge -1$ such that $L_i\neq M_i$. Then $M_i$, being the strictly increasing union of $Q_\chi$-submdules $\bigcup_{n\ge 0}u^{-n}L_i$, is not compactly generated as a $Q_\chi$-module. We deduce $\chi\in\Gamma_Q(M_i)\subseteq\Gamma_Q(M)$ (because $M_i$ is a quotient of $M$), and we are done. 
\end{proof}

\begin{thm}\label{chisous}
Let $Q$ be a compactly generated abelian group. Let $M$ be a LC module and $N$ a closed submodule. Then $\Gamma_A(N)\subseteq\Gamma_A(M)$.
\end{thm}
\begin{proof}
The proof relies on two main particular cases: when $N$ is open (Lemma \ref{opega}), and when $M$ is sheer (Lemma \ref{gammasheer}).

The result is trivial if $M$ is not compactly generated, since then $\Gamma_Q(M)=Q^*$. Assume that $M$ is compactly generated. By Corollary \ref{maxcp}, $\Ell(M)$ is compact. Note that $\Ell(N)\subseteq\Ell(M)$. We have the closed embedding $N/\Ell(N)\subseteq M/\Ell(M)$, and since $\Gamma_Q(-)$ does not change when modding out by a compact submodule, we can suppose that $\Ell(N)=\Ell(M)=\{0\}$. Thus $M$ is lower-sheer.

Write $\Omega=\Omega^\sharp(M)$, so $\Omega$ is sheer. Then $\Gamma_Q(N)\subseteq\Gamma_Q(\Omega\cap N)\cup\Gamma_Q(N/(\Omega\cap N))$ by Lemma \ref{extgam}. So it is enough to show that the latter two are contained in $\Gamma_Q(M)$. Indeed, first, we have $\Gamma_Q(\Omega\cap N)\subseteq \Gamma_Q(\Omega)\subseteq\Gamma_Q(M)$: here the first inclusion follows from Lemma \ref{gammasheer}, and the second from Lemma \ref{opega}. Second, we have $\Gamma_Q(N/N\cap\Omega)=\Gamma_Q(N+\Omega/\Omega)\subseteq\Gamma_Q(M/\Omega)\subseteq \Gamma_Q(M)$. Here the first inclusion follows from Lemma \ref{opega}, and the second one is trivial.
\end{proof}

\subsection{Characterization of compact presentability}

\begin{thm}
Let $G$ be a compactly generated locally compact metabelian group. Then $G$ is compactly presentable if and only if $\Gamma(G)$ contains no line (i.e.\ $\Gamma(G)\cap -\Gamma(G)=\{0\}$).
\end{thm}
\begin{proof}
We can suppose that $G$ is totally disconnected, since replacing $G$ by $G/G^\circ$ does not affect being compactly presentable, and does not modify $\Gamma(G)$ (it would modify its open complement!).
Let $M$ be a closed abelian normal subgroup of $G$ such that $Q=G/N$ is abelian.

By Proposition \ref{ratpo}, $\Gamma(G)\cap -\Gamma(G)$ is a rational polyhedral cone. Hence if not reduced to $\{0\}$, it contains a rational half-line; some element in this line is a homomorphism $\chi$ onto $\Z$; both $\chi$ and $-\chi$ belong to $\Gamma(G)$. Then $\Gamma(G)=\Gamma_Q(M)$ by Lemma \ref{eqmet}, and hence $\pm\chi\notin\Gamma_Q(M)$. The first Bieri-Strebel theorem \cite{BS78}, in its locally compact version \cite[Cor.\ 8.C.4]{CH} implies that $G$ is not compactly presentable.

Conversely suppose that $\Gamma(G)\cap -\Gamma(G)=\{0\}$ and let us show that $G$ is compactly presentable. So $M$ is a compactly generated $Q$-module, and by Corollary \ref{maxcp} $\Ell(M)$ is compact. We can mod out and suppose that $\Ell(M)$ is reduced to zero. Since $M$ is totally disconnected and $\Ell(M)$ is zero, Theorem \ref{sheerdec1} ensures that the open submodule $N=\Omega^\flat(M)$ is polycontractable. 

By Proposition \ref{monotypic}, there exists $k\ge 0$ such that $N$ is the direct product of $k$ nonzero monotypic contractable modules. We prove that $G$ is compactly presentable by induction on $k$.

If $k=0$, we have $N=\{0\}$ and thus $G$ is discrete. By Bieri-Strebel's characterization of finitely presentable metabelian groups \cite{BS}, $G$ is finitely presentable.

%

If $k\ge 1$, let $V$ be one of these monotypic submodules By induction, $G/V$ is compactly presentable. Let $\K$ be the simple quotient of $V$. By Corollary \ref{sico}, $\K$ has a structure of non-discrete totally disconnected locally compact field, such that the $Q$-action is given by a homomorphism $\phi:G\to\K^\times$ (factoring through $Q=G/M$); define $\chi=-\log|\phi|$. Thus $\Gamma_Q(V)$ is the half-line generated by $\chi$ (Lemma \ref{gamono}). Fix $u\in G$ with $\chi(u)>0$ equal to the positive generator of $\chi(G)$; thus $0<|\phi(u)|<1$.

By Lemma \ref{gamono}, there is a compact open $Q_{\chi}$-submodule $W$ of $N$. Thus $uW\subseteq W$ and $\bigcup_{n\ge 0}u^{-n}W=N$. Consider a symmetric subset $S\subset\Ker(\chi)$ with non-empty interior such that $S\cup\{u\}$ generates $G$. Let us check the assumptions of Lemma \ref{condn2}: those written in Lemma \ref{condn2} are fulfilled; among those in Lemma \ref{condn}, (\ref{rel11}), (\ref{rel12}) hold, and let us check (\ref{rel2}), namely that $uSu^{-1}$ is contained in the $u^{-}$-subgroup $H$ generated by $S$. Indeed, clearly $H$ is open; we have $G=\bigcup_{n\ge 0} u^nHu^{-n}$ (ascending union). Since $-\chi\notin\Gamma(G)$, the normal subgroup $\overline{[G,G]}$ is compactly generated as a $G_{-\chi}$-group, and hence so is $H$ (since $H/ \overline{[G,G]}$ is a compactly generated abelian group). Therefore the above ascending union stabilizes, which means that $uHu^{-1}=H$, and hence that $G=H$. Therefore Lemma \ref{condn2} implies that $G$ is compactly presentable.
\end{proof}


\section{A few counterexamples}

The following gathers ``counterexamples" to the above results, when one assumption is relaxed. Especially, we consider cases when the acting group $A$ is not compactly generated.

Here we fix $A$ discrete free abelian of countable rank with basis $(e_i)_{i\ge 1}$. In each case, $p$ is a fixed prime (and $A$ could also be replaced with $A/pA$).

\begin{ex}[Theorem \ref{eclo} fails without compact generation of $A$.]
Consider the additive group $M$ of $\F_p\lp t\rp$. Let $A$ act as follows: \[e_i\cdot\sum_n a_nt^n=\sum_{n\neq\pm i}a_nt^n+a_it^{-i}+a_{-i}t^i.\]
In other words, $e_i$ exchanges $t^i$ and $t^{-i}$, fixes other monomials, and is extended by continuity.
Then $\Ell_A(M)=\F_p[t,t^{-1}]$ is not closed. 
\end{ex}

\begin{ex}[Corollary \ref{wcompact-i} fails without compact generation of $A$.]
Let $I$ be the ideal of $\Z/p\Z[A]$ generated by all $e_i$. Consider the discrete module $M=\Z/p\Z[A]/I^2$. It is finitely generated, i.e., compactly generated. Then $\Ell_A(M)=I/I^2$ is infinite, i.e., not compact.
\end{ex}


\begin{ex}[Theorem \ref{cafi} fails without compact generation of $A$.]
Let $(X_n)_{n\ge 1}$ be a partition of $\Z$ such that for each $n$, $X_n$ is infinite and $X_n\cap\N_{\ge 1}=\{n\}$. Let $\sigma_n$ be a permutation of $\Z$ acting as a single infinite cycle on $X_n$ and identity elsewhere. Let $M$ be the additive group  of $\F_p\lp t\rp$. Let $A$ act as follows:
\[e_i\cdot\sum_n a_nt^n=\sum_n a_nt^{\sigma_i(n)}.\]
This is well-defined because $\sigma_i$ maps every lower-bounded subset to a lower-bounded subset, and continuous because $\sigma_i$ is identity near $+\infty$. Since the $\sigma_i$ have pairwise disjoint supports, they commute and thus this defines an action of $A$. Then, as a $A$-module, $M$ is residually purely discrete: indeed it embeds into the product $\prod_n\F_p^{(X_n)}$. Hence, $\Omega^\sharp(M)=\{0\}$. But $M$ is not discrete. Thus $\Omega^\sharp(M)$ is not open. 

If $N$ is the Pontryagin dual, then $N$ is the closure of $\Ell^\flat(N)$ but $N$ is not compact. 
\end{ex}


\begin{ex}[Non-compactly generated metabelian group $G$ in which $\Ell(G)$ is not closed]
This is a classical example \cite[\S 6, Example 1]{WY}. Let $(p_i)$ be a sequence of odd primes (it can be constant or injective), and $G=\bigoplus_i \F_{p_i}\rtimes \prod_i \F_{p_i}^*$ (as an abstract group, this can be viewed as a subgroup of $\prod_i \F_{p_i}\rtimes\F_{p_i}^*$, namely consisting of those sequences for which the additive term is eventually zero). Then $G$ is metabelian and $\Ell(G)$ is not closed and actually dense, namely equal to $\bigoplus_i \F_{p_i}\rtimes\F_{p_i}^*$.  
\end{ex}

\begin{ex}[Theorem \ref{cgmmax}(\ref{tm2}) and Theorem \ref{sheerdec1} fails without compact generation of $A$]
Let $M$ be the group of adeles, namely the sequences $(n_p)$, with $n_p\in\Q_p$ and $n_p\in\Z_p$ for all but finitely many $p$, with compact open subgroup $\prod_p\Z_p$. Define $G=M\rtimes\prod_p\langle p\rangle$ with the natural action. Then $[G,G]$ equals the dense subgroup $\bigoplus_p\Q_p$ of $M$, hence is not closed, although $\Ell(G)$ is trivial (thus Theorem \ref{cgmmax}(\ref{tm2}) ``fails" here).

Also, this module is sheer, has no nonzero Euclidean or amorphic submodule, but is not polycontractable; hence Theorem \ref{sheerdec1} ``fails" too.
\end{ex}

\begin{ex}
Theorem \ref{cgmmax}(\ref{tm2}) is not true without the ``totally disconnected" assumption. Indeed, fix a dense proper subgroup $\Lambda$ of $\R$, e.g., $\Z[\sqrt{2}]$. Consider the semidirect product $G=(\R\times\Lambda)\rtimes\Z$ where the action is by powers of the automorphism $(x,t)\mapsto (x+t,t)$. Then $[G,G]$ is the non-closed subgroup $\Lambda\times\{0\}$ of $\R\times\Lambda$.
\end{ex}


\begin{thebibliography}{KM98b}

\bibitem[Ab]{Ab72} H. Abels. Kompakt definierbare topologische Gruppen. Math. Ann. 197 (1972), 221--233. 



\bibitem[BG]{BG} R. Bieri and J. Groves. The geometry of the set of characters induced by valuations. J. Reine Angew. Math. 347, 168--195 (1984).

\bibitem[BNS]{BNS} R. Bieri, W. Neumann and R. Strebel. A geometric invariant of discrete groups. Invent. Math. 90 (1987) 451--477.

\bibitem[BS1]{BS78} R. Bieri and R. Strebel. Almost finitely presented soluble groups. Comment. Math. Helv. 53 (1978) 258--278. 

\bibitem[BS2]{BS} R. Bieri and R. Strebel. Valuations and finitely presented metabelian groups. Proc. London Math. Soc. (3) 41 (1980) 439--464. 


\bibitem[BW]{BW} U. Baumgartner, G. Willis. Contraction groups and scales of automorphisms of totally disconnected locally compact groups. Israel J. Math. 142 (2004), 221--248.



\bibitem[CGl]{CGl} H. Corson and I. Glicksberg. Compactness in Hom$(G,H)$. Canad. J. Math. 22 (1970), 164--170.

\bibitem[CGu]{CGu} J-P. Conze, Y. Guivarc'h. Remarques sur la distalit\'e dans les espaces vectoriels. C. R. Acad. Sci. Paris Ser. A 278 (1974), 1083--1086.

\bibitem[CH]{CH} Y. Cornulier, P. de la Harpe. Metric geometry of locally compact groups. Tracts in Math 25. European Math. Soc, 2016.

\bibitem[C15]{C15} Y. Cornulier. Commability and focal locally compact groups. Indiana Univ. Math. J. 64(1) (2015) 115--150.


\bibitem[CT]{CT} Y. Cornulier, R. Tessera. Geometric presentations of Lie groups and their Dehn functions. Publ. Math. IHES 125(1) (2017), 79--219.



\bibitem[GW]{GW} H. Glockner, G. Willis. Classification of the simple factors appearing in composition series of totally disconnected contraction groups, J. Reine Angew. Math. 141--169 (2010).


\bibitem[Si]{Si} E. Siebert. Contractive automorphisms on locally compact groups, Math. Z. 191 (1986), 73--90.


\bibitem[Wil]{Wil} G. Willis. The structure of totally disconnected, locally compact groups. 
Math. Ann. 300 (1994), no. 2, 341--363. 

\bibitem[WY]{WY} T. S. Wu, Y.K. Yu. Compactness properties of topological groups, Michigan Math. J. 19
(1972), no. 4, 299--313. 

\end{thebibliography}
\end{document}